\documentclass[3p]{elsarticle}
\usepackage[utf8]{inputenc}
\usepackage[table,dvipsnames]{xcolor}

\usepackage{graphicx} % Required for inserting images
\usepackage{multirow}
\usepackage{amsmath}
\usepackage{amssymb}
\usepackage{algorithm}
\usepackage{algorithmic}

\usepackage{caption}
\usepackage{subcaption}
\usepackage{amsthm}
\newtheorem{proposition}{Proposition}
\newtheorem{remark}[proposition]{Remark}
\usepackage{todonotes}
\usepackage{wrapfig}
\usepackage{hyperref}
\usepackage{chngcntr}
\counterwithin{figure}{section}
\counterwithin{table}{section}
\counterwithin{algorithm}{section}
\counterwithin{proposition}{section}
\counterwithin{equation}{section}

\usepackage{natbib}
\newcommand{\authoryearcite}[1]{[\citeauthor*{#1}~(\citeyear{#1})]}

\newcommand{\revision}[1]{#1}

\usepackage{epstopdf}
\usepackage{stmaryrd} %llbracket, rrbracket
\usepackage{fixmath} %mathbold
\usepackage{pgf}
\usepackage{wasysym}

\newcommand{\setstyle}[1]{{\mathbb #1}}

\newcommand{\norm}[2]{{\| #1 \|}_{#2}}

\def\setR {\setstyle{R}}

\def\setN {\setstyle{N}}

\makeatletter
\def\ps@pprintTitle{
 \let\@oddhead\@empty
 \let\@evenhead\@empty
 \def\@oddfoot{}
 \let\@evenfoot\@oddfoot}
\makeatother

\title{Data-driven geometric parameter optimization for PD-GMRES}

\author[inst]{Lennart Duvenbeck}
\ead{lennart.duvenbeck@ians.uni-stuttgart.de}
\author[inst]{Cedric Riethm\"uller\corref{cor1}}
\ead{cedric.riethmueller@ians.uni-stuttgart.de}
\author[inst]{Christian Rohde}
\ead{christian.rohde@ians.uni-stuttgart.de}

\affiliation[inst]{organization={
Institute of Applied Analysis and Numerical Simulation,
University of Stuttgart},%Department and Organization
            city={70569 Stuttgart}, 
            country={Germany} \\
            (Emails:  $\{\text{lennart.duvenbeck, cedric.riethmueller, christian.rohde}\}$@ians.uni-stuttgart.de)
            }
            
\cortext[cor1]{Corresponding author}

\begin{document}

\begin{abstract}
Restarted GMRES is a robust and widely used iterative solver for linear systems. The control 
of the restart parameter is a key task to accelerate convergence and to prevent the  well-known stagnation phenomenon. We focus on the Proportional-Derivative GMRES (PD-GMRES), which has been derived using control-theoretic ideas in  \authoryearcite{CuevasNunez2018} as a versatile method for modifying the restart parameter.\\
Several variants of a quadtree-based geometric optimization approach are proposed 
to find a best choice of PD-GMRES parameters.  We show that the optimized 
PD-GMRES performs well across a large number of matrix types and we observe superior performance as compared to major other GMRES-based iterative solvers. Moreover, we propose an extension of the PD-GMRES algorithm to further improve performance by controlling the range of values for the restart parameter.
%In addition to making a generic recommendation for PD-GMRES parameters, the geometric approach can also be used to help optimize arbitrary two-dimensional functions on bounded domains. With application to PD-GMRES this allows for targeted optimization of the control schemes parameters which results in highly superior performance compared to all other GMRES-based iterative solvers.
\end{abstract}

\begin{keyword}
iterative solver \sep PD-GMRES \sep  parameter optimization \sep quadtree \sep data-driven methods
% %% MSC codes here, in the form: \MSC code \sep code
% %% or \MSC[2008] code \sep code (2000 is the default)
\end{keyword}

\maketitle

\section{Introduction}
The Generalized Minimal RESidual (GMRES) method is a well-established algorithm for iteratively finding   approximations for  the solution $x\in \setR^n$, $n\in \setN$, of the system of linear equations 
\begin{equation}\label{lineq}
Mx = b,
\end{equation}
with $M\in \setR^{n \times n} $, $\det(M) \neq 0$, and  $b\in \setR^n$   \cite{Saad2003,GMRES}. GMRES belongs to the class of Krylov subspace methods. It works by finding an approximate solution $x_\ell \in x_0 + K_\ell$, where $K_\ell = \text{span}\{ r_0, Mr_0, \dots, M^{\ell-1}r_0\}$ is the $\ell$-th Krylov subspace for the residual $r_0 = Mx_0 - b$ of the initial guess $x_0\in \setR^n$, such that the Euclidean norm of the residual $ \norm{r_\ell}{2} = \norm{Mx_\ell - b}{2}$ is minimized. This ensures that the residual norm does not increase for any step of the iteration. The iterate $x_\ell$ itself is found by solving a $(\ell \times \ell)$-least-squares problem.\\
While GMRES is guaranteed to terminate after $n$ steps with $x_n =x$, this is usually not practical as both the storage and runtime cost of the least-squares problem become unreasonably high. Therefore restarted variants of GMRES are used in practice instead. Such methods use a restart parameter $m \in \setN$ and calculate the (inner) iterations of GMRES only until the $m$-th Krylov subspace. The approximate solution after that step is then taken as the new initial guess and a new Krylov subspace is constructed from the residual to this approximate solution. This method, referred to as GMRES($m$), has the advantage that only $m$ steps of GMRES are performed per restart which constrains the storage and time requirements. However, restarting GMRES in this manner reduces the rate of convergence as information from previous restarts is discarded and only small Krylov subspaces are searched. At its worst, GMRES($m$) can exhibit very slow convergence or even stagnation \cite{GMRES-Stagnation,Meurant2014,Simoncini2010,GMRES-Stagnation2}.  This means that the approximate solution does not improve at all during the search in the $m$-dimensional Krylov subspace, in which case the algorithm does not converge.

As it is difficult to know a-priori how the restart parameter $m$ should be chosen to guarantee a fast rate of convergence, modifications of GMRES($m$) have been employed. One category involves modification of the Krylov subspace by augmentation with elements containing information about previous restarts, by enriching with spectral information and by deflation \cite{GMRES-Augmented2, Cabral2020, Eiermann, GMRES-Augmented1,Simoncini2007}. Another type of modification uses adaptive methods to modify $m$ between restarts of GMRES to avoid the problems which can arise from a poorly chosen fixed restart parameter. Such adaptive GMRES variants adjust $m$ depending on the rate of convergence to improve the runtime of the iteration by increasing $m$ when it is necessary to avoid stagnating or near-stagnating behavior, while avoiding increasing $m$ too much so that solving the least-squares problem remains feasible \cite{Baker,Joubert}.

The specific method we analyze is PD-GMRES \cite{CuevasNunez2018}, which relies on  a Proportional-Derivative controller to determine the restart parameter $m$ based on the residual norms of the iteration after previous restarts. Such a controller is a specific example of a more general class of proportional-integral-derivative (PID) controllers which are well established, error-based tools in control theory \cite{PID-Control2,PID-Control1}.

The PD-Controller used in PD-GMRES contains five free parameters which can be set to provide an optimal rate of convergence (see Section \ref{subsec:parameters} below). For any individual matrix $M$ (or set of matrices) this creates a multidimensional optimization problem with the goal of finding optimal values for the five parameters such that PD-GMRES exhibits fastest performance. Instead of using a direct optimization method such as gradient descent or a derivative-free method such as Nelder-Mead \cite{NelderMead}, we employ a quadtree-based geometric method in order to find local minima for the PD-GMRES runtime, while also gathering a low-resolution overview on the performance across the parameter space (achieving similar goals as~\cite{GeometricApproach}). This is useful to avoid the selection of low-quality local minima. It is in fact necessary for evaluating a set of matrices and determining optimal average performance.

Quadtrees are tree data structures in which each inner node has exactly four children. They are used in  multiple applications where a two-dimensional space needs to be partitioned efficiently \cite{Quadtree}. Such applications for quadtrees (and their three-dimensional equivalent, octrees) include image compression \cite{Quadtree-ImageCompression3,Quadtree-ImageCompression1,Quadtree-ImageCompression2}, video coding \cite{Quadtree-VideoCoding2,Quadtree-VideoCoding1}, collision detection \cite{Quadtree-CollisionDetection1}, mesh generation \cite{Quadtree-MeshGeneration2,Quadtree-MeshGeneration1} and adaptation \cite{Ebeida2008,Huo2019,Jaillet2022}, spatial indexing \cite{Quadtree-SpatialIndexing2,Quadtree-SpatialIndexing1}, terrain triangulation and visualization \cite{Quadtree-Terrain1,Quadtree-Terrain2}, and connected component labeling \cite{Quadtree-CCL1}.

\revision{
In summary, the main contributions of this paper are the following.
\begin{itemize}
 \item We present a general quadtree-based method for evaluation and optimization of functions on rectangular two-dimensional domains.
 \item We employ a heuristic runtime estimate for GMRES-based iterative solvers to avoid volatile runtime measurements.
 \item We develop an optimization procedure based on the quadtree method in which PD-GMRES parameters are grouped and optimized alternately to improve PD-GMRES performance. This method can be applied to arbitrary matrix sets to find optimized parameters.
 \item We apply our optimization procedure to a general set of training matrices and validate the parameters with a set of test matrices. From this we make recommendations for generic PD-GMRES parameter sets which can be used as default values when utilizing PD-GMRES.
\end{itemize}
}

The structure of this work is as follows. In Section \ref{sec:pdgmres} the PD-GMRES algorithm from \cite{CuevasNunez2018} is reviewed. We further show some theoretical results concerning the convergence behavior depending on specific parameter ranges. In Section \ref{sec:optim} we introduce the quadtree-based geometric approach and adapt it to the optimization of the PD-GMRES parameters. This includes the use of averaged heuristic runtime estimates and the grouping of PD-GMRES parameters in pairs to apply our optimization approach, reducing a five-dimensional optimization problem to two dimensions. Section \ref{sec:numer} comprises numerical experiments in which we investigate the application of our geometric approach to optimize PD-GMRES parameters for a set of training matrices. The optimization procedure yields an optimized PD-GMRES that exhibits very good convergence rates for a significant number of matrices. In Section \ref{subsec:extend} we propose an extension of the PD-GMRES algorithm using an additional parameter $m_{\max}$ and showcase that this extended algorithm can provide better performance for a matrix for which the rate of convergence is very slow. We end with the concluding  Section \ref{sec:Conclusion}.

The code and data for this work are publicly available \cite{DARUS-4812_2025}.

\section{The PD-GMRES algorithm}\label{sec:pdgmres}
% PD-GMRES is an iterative solver for a system of linear equations $Mx = b$. 
PD-GMRES has been introduced in \cite{CuevasNunez2018} as an adaptive extension of GMRES($m$).
Starting with an initial vector $x_0$ to $\eqref{lineq}$ and corresponding initial residual $r_0 = Mx_0 - b$ as well as an initial restart parameter $m_1 = m_{\text{init}} \in \setN$, PD-GMRES calls single iterations of GMRES($m_j$) which provide $x_j$ and the corresponding residual $r_j$ for the next restart. The next restart parameter $m_{j+1}$ is determined by the Proportional-Derivative (PD) controller based on the residual norms of the iteration after the previous two restarts (Algorithm \ref{alg:pd_controller}), and GMRES($m_{j+1}$) is then executed on the previous iteration vector $x_j$ to get the next iteration vector $x_{j+1}$. The main PD-control equation is given by 
\begin{equation}
\label{eq:pdcontrol}
    m_{j+1} = m_j + \left\lfloor P_j(\alpha_p) +  D_j(\alpha_d) \right\rfloor.
\end{equation}
In \eqref{eq:pdcontrol} we have the  proportional term
\begin{equation}
 P_j(\alpha_p) = \alpha_p \cdot\frac{\|r_j\|}{\|r_{j-1}\|}
\label{eq:proportional_term}
\end{equation}
and the  derivative term
\begin{equation}
 D_j(\alpha_d) = \alpha_d \cdot\frac{\|r_j\|-\|r_{j-2}\|}{2\|r_{j-1}\|}.
\label{eq:derivative_term}
\end{equation}
Their  relative weight  is balanced by the parameters $\alpha_p, \alpha_d \in \setR$. The iteration continues until the residual norm $\|r_j\| $ is less than a specified tolerance.

\begin{algorithm}
\caption{PD-controller}
\label{alg:pd_controller}
\begin{algorithmic}
\STATE \hspace*{-\algorithmicindent}\textbf{Parameters:} $m_{\text{init}}, m_{\min}, m_{\text{step}}, \alpha_p, \alpha_d$
\REQUIRE{$m_j, r_j, r_{j-1}, r_{j-2}, c$}
\ENSURE{$m_{j+1}, c$}
\IF{$j = 0$}
\STATE{$m_{j+1} = m_{\text{init}}$}
\ENDIF
\IF{$j = 1$}
\STATE{$m_{j+1} = m_j + \left\lfloor P_j(\alpha_p) \right \rfloor$}
%+  D_j(\alpha_d) \right\rfloor %\left\lfloor\alpha_p \cdot \frac{\|r_j\|}{\|r_{j-1}\|} \right \rfloor$}
\ENDIF
\IF{$j \geq 2$}
\STATE{$m_{j+1} = m_j + \left\lfloor P_j(\alpha_p) +  D_j(\alpha_d) \right\rfloor $} %\left\lfloor \alpha_p \cdot\frac{\|r_j\|}{\|r_{j-1}\|} +  \alpha_d \cdot\frac{\|r_j\|-\|r_{j-2}\|}{2\|r_{j-1}\|}\right\rfloor$}
\ENDIF
\IF{$m_{j+1} < m_{\min}$}
\STATE{$c = c+1$ \hfill\COMMENT{Global counter of number of times $m_{j+1}$ has fallen below $m_{\min}$}}
\STATE{$m_{j+1} = m_{\text{init}}+c\cdot m_{\text{step}}$}
\ENDIF
\end{algorithmic}
\end{algorithm}

As the control equation~\eqref{eq:pdcontrol} can cause the restart parameter $m_{j+1}$ to become too small to be useful or even negative, an additional parameter $m_{\min} \in \setN$ is introduced as the minimal allowed restart parameter. If $m_{j+1}$ falls below $m_{\min}$ it resets to $m_{\text{init}} + c \cdot m_{\text{step}}$ (as simply resetting to $m_{\text{init}}$ would create a risk of stagnation). Here, $m_{\text{step}} \in \setN$ is an additional parameter which determines by how much this reset value increases and $c$ is a counter for the resets initialized at $0$.

We consider the behavior of the PD-controller depending on the sign of the parameters $\alpha_p$ and $\alpha_d$. For $\alpha_p < 0, \alpha_d > 0$, both the proportional and the derivative term are negative, causing the restart parameter to decrease each iteration as the terms are rounded down. In combination with the reset condition, this creates a behavior where the restart parameter decreases until it reaches a value less than $m_{\min}$ and then jumps to a value higher than $m_{\text{init}}$ with that value increasing with each reset. This allows PD-GMRES to overcome stagnation. In the case of stagnation we have $\frac{\|r_j\|}{\|r_{j-1}\|} \approx 1$ and $\frac{\|r_j\|-\|r_{j-2}\|}{2\|r_{j-1}\|} \approx 0$. Therefore, as long as $\alpha_p \neq 0$ the restart parameter will decrease after each restart and, by jumping to the increasing reset value $m_{\text{init}} + c \cdot m_{\text{step}}$, will eventually reach the matrix dimension $n$, at which point GMRES converges.

\begin{proposition}
	For PD-GMRES with $\alpha_p < 0$ and $\alpha_d \geq 0$, there exists a number $N \in \{0,\ldots,n(n+1)/2\}$ such that $x_N = x = M^{-1}b$.
	\label{prop:1}
\end{proposition}
\begin{proof}
	Given $\alpha_p < 0$ and $\alpha_d \geq 0$, the proportional term $P_j(\alpha_p)$ in the PD-controller is always negative and the derivative term $D_j(\alpha_d)$ non-positive. As the controller rounds down, it follows that $m_{j+1} \leq m_j - 1$ at each restart. After at most $m_j - m_{\min}$ restarts, the reset condition is reached which increases the restart parameter. This process repeats for at most $n(n+1)/2$ restarts, until the restart parameter reaches the dimension $n$ of the matrix $M$, at which point GMRES finds the exact solution $x$.
\end{proof}

\begin{remark}
While Proposition~\ref{prop:1} does provide a guarantee that PD-GMRES will always converge if $\alpha_p < 0$ and $\alpha_d \geq 0$, this is usually not practically relevant as it requires $\mathcal{O}(n^2)$ restarts to actually reach the dimension $n$ of the matrix $M$. However, as it is also usually not necessary for the restart parameter to reach matrix dimension in order to have a good rate of convergence, the guaranteed increase of the restart parameter is still a useful result. Moreover, it provides theoretical justification for choosing a negative value for $\alpha_p$ and a positive value for $\alpha_d$.
\end{remark}

For $\alpha_p > 0, \alpha_d > 0$, the restart parameter increases directly in the case of stagnation, similarly guaranteeing convergence as the restart parameter will reach $n$. In this parameter range it is possible for the proportional and derivative terms in the control equation~\eqref{eq:pdcontrol} to sum up to $0$, resulting in $m_{j+1} = m_j$. In that case, the ratio $\frac{\|r_{j-2}\|}{\|r_j\|}$ can be directly determined from $\alpha_p$ and $\alpha_d$.

\begin{proposition}
    Consider a PD-GMRES iteration with $\alpha_p > 0$ and $\alpha_d > 0$ at step $j$. If the proportional \eqref{eq:proportional_term} and derivative \eqref{eq:derivative_term} terms in the controller add to zero, then $m_{j+1} = m_j$ and $\frac{\|r_{j-2}\|}{\|r_j\|} = 1 + \frac{2\alpha_p}{\alpha_d}$.
    \label{prop:2}
\end{proposition}
\begin{proof}
    The statement follows directly from the control equation~\eqref{eq:pdcontrol}.
    % \begin{equation*}
    % 	\alpha_p \frac{\|r_j\|}{\|r_{j-1}\|} + \alpha_d \frac{\|r_j\|-\|r_{j-2}\|}{2\|r_{j-1}\|} = 0.
    % \end{equation*}
\end{proof}

In situations where Proposition~\ref{prop:2} applies, the rate of convergence is particularly good if $\alpha_p$ is large or $\alpha_d$ is small. If $\alpha_p = \alpha_d$, the residual norm improves by a factor of $3$ every two restarts. If the restart parameter $m_{j+1}$ is decreased by the controller, then $|D_j(\alpha_d)| > |P_j(\alpha_p)|$
%the absolute value of the derivative term must be greater than that of the proportional term
which results in a rate of convergence that is even better than $1 + \frac{2\alpha_p}{\alpha_d}$. This provides some theoretical justification for choosing a positive value for both $\alpha_p$ and $\alpha_d$.

\begin{remark}
It must be noted that the case that the restart parameter does not change is not directly addressed by Proposition~\ref{prop:2}, as it is possible for $m_{j+1} = m_j$ while $0 < P_j(\alpha_p) + D_j(\alpha_d) < 1$. In that case the floor function in the control equation~\eqref{eq:pdcontrol} rounds down, and the rate of convergence is slightly worse than described in Proposition~\ref{prop:2}. Because of this, the systematic error caused by the rounding in the controller becomes undesirable. To reduce this error, we adjust the control algorithm by storing the remainder of the rounding operation and re-adding it during the next iteration. As the rounding was actually used in Proposition~\ref{prop:1}, the modification of storing and re-adding the rounding error is only applied if $\alpha_p > 0$.
\end{remark}

Due to the theoretical convergence results found in Propositions \ref{prop:1} and \ref{prop:2} for $\alpha_d > 0$ with $\alpha_p < 0$ and $\alpha_p > 0$ respectively, these are the two parameter ranges which we evaluate in our numerical experiments in Section~\ref{sec:numer}.
% \ref{subsec:2q}

\section{Geometric optimization approach} \label{sec:optim}
For the purpose of determining the PD-GMRES parameters $m_{\text{init}}$, $m_{\min}$, $m_{\text{step}}$, $\alpha_p$, $\alpha_d$ that minimize the PD-GMRES runtime, we describe a general quadtree-based algorithm for finding minima in a bounded two-dimensional domain which we then adapt to optimize the PD-GMRES parameters in pairs.

\subsection{The quadtree method} \label{sec:quadtree}
Given a rectangular domain $D \subset \setR^2$ and a function $f : D \rightarrow \setR$, our quadtree approach can be considered a reverse analogue of using quadtrees for image compression. While in image compression a quadtree approach is used to reduce the amount of data stored while still retaining enough important information to approximate the full image, the goal of our approach is to understand the behavior of $f$ and find minima in the domain $D$ while calculating as few actual evaluations of $f$ as possible. A simple method for this would be to evaluate $f$ on a uniform grid. Such a method is not well suited for finding local minima and likely involves many unnecessary evaluations of $f$ in less useful locations. Another option would be to take a steepest descent method, starting from some quantity of distributed points. This option generates local minima, but might not provide a good overview of the behavior of $f$ in the domain. The quadtree approach can achieve both goals by evaluating the function on its domain with varying resolution.

The basic idea of the algorithm is to construct a matrix $Q$,
% of resolution $\setR^{2^d \times 2^d}$
the entries of which represent the value of $f$ at points in the domain $D$ (formally, we map $Q$ to the domain $D$ by an affine function $\phi$ in the sense that each entry of $Q$ is mapped to some point in $D$). We then break the domain down into four quadrants and evaluate $f$ at a single point per quadrant, with the evaluations being set as the values for corresponding blocks in $Q$. Some of the quadrants are then subdivided further with a single evaluation being calculated per new subquadrant. At each step of this iteration only a portion of the smallest quadrants is subdivided further by some \textbf{criterion C}, which results in $Q$ serving as a visualization of the behavior of $f$ on $D$, with high resolution in certain areas and lower resolution in others.
\begin{algorithm}
\caption{Quadtree based optimization}
\label{alg:quadtree}
\begin{algorithmic}
\REQUIRE{$f : D \rightarrow \setR,\ d \in \setN,\ \textbf{criterion C},\ \phi : Q \rightarrow D$}
\ENSURE{$Q \in \setR^{2^d \times 2^d}$}
\STATE{Evaluate $f$ for a single value pair $(u^\ast,v^\ast) \in D$}
\STATE{Initialize $Q \in \setR^{2^d \times 2^d}$ with the value $f(u^\ast,v^\ast)$ for all entries}
\STATE{Initialize a set $B$ containing the single $(2^d \times 2^d)$-block of $Q$}
\FOR{$i = 1$ \TO $d$}
\STATE{Choose a subset $B' \subseteq B$ by \textbf{criterion C} and set $B = \emptyset$}
\FORALL{blocks $\beta' \in B'$}
\STATE{Subdivide $\beta'$ into four blocks of size $(2^{d-i} \times 2^{d-i})$ and add them to $B$}
\ENDFOR
\FORALL{blocks $\beta \in B$}
\STATE{Evaluate $f$ for one value pair $(u_\beta, v_\beta) \in \phi(\beta)$}
\STATE{Assign the value $f(u_\beta, v_\beta)$ to all entries of block $\beta$ in $Q$}
\ENDFOR
\ENDFOR
\end{algorithmic}
\end{algorithm}

The general principle of the quadtree approach is shown in Algorithm~\ref{alg:quadtree}. Here, $d \in \setN$ represents the maximal depth to which the parameter space is evaluated. The resulting matrix $Q$ consists of homogeneous quadratic blocks of sizes $2^k$ with $0 \leq k \leq d$. A visualization of an example quadtree structure is shown in Figure~\ref{fig:qt_vis}.
%The implementation in Matlab makes use of a sparse matrix which indexes the block locations and sizes.

\newpage % Wrapfigure on new page for better layout
\begin{wrapfigure}{r}{0.4\textwidth}
\includegraphics[width=0.38\textwidth]{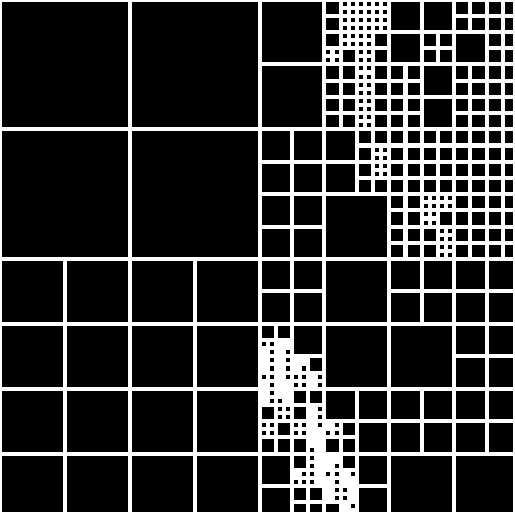}
\caption{Quadtree visualization.}
\label{fig:qt_vis}
\end{wrapfigure}

In Algorithm~\ref{alg:quadtree} it is not specified by which \textbf{criterion C} the blocks to be subdivided further are chosen. Here, several different options are possible. As minimizing $f$ is a key goal, it is natural to select the blocks where the evaluation of $f$ is minimal. Then, the simplest option is to impose a fixed ratio of blocks to be subdivided for each level which allows explicit control of the number of function evaluations that are performed when creating the quadtree.
The total number of function evaluations which are performed is
\begin{equation*}
    1 + \sum_{i=1}^d \prod_{j=1}^i 4 \cdot s_j,
\end{equation*}
given fixed block ratios $s_j$ for $1 \leq j \leq d$. This option, which we use in this work, corresponds to a constrained breadth-first search. Similarly a depth-first search is also possible. Alternative methods for choosing which blocks to divide are of course possible, such as determining $s$ dynamically during the iteration which could in principle help avoid unnecessary evaluations. It is also possible to base the subdivisions not on minimizing $f$ but on other properties. One such option would be to consider the differences in the values of neighboring blocks and subdividing those where this difference is large.

\subsection{Application to PD-GMRES}
In this section we apply the quadtree approach from Section~\ref{sec:quadtree} to minimize the runtime of PD-GMRES by optimizing the algorithm's free parameters $m_{\text{init}}$, $m_{\min}$, $m_{\text{step}}$, $\alpha_p$, $\alpha_d$ in pairs.
%This can be done for a single matrix or a set of matrices.

\subsubsection{A heuristic runtime estimate}
\revision{
As runtime measurements are volatile, in particular relative to the differences one can expect for slight parameter changes in PD-GMRES, we instead use a heuristic function which can deterministically calculate a runtime estimate. A PD-GMRES iteration consists of individual iterations of GMRES($m_k$) with restart parameters $(m_k)$. As a basis for our estimate we use $\sum_k h(m_k)$, where the function $h : \setN \rightarrow \setR_+$ gives a runtime estimate $h(m_k)$ for one iteration of GMRES($m_k$). We choose $h$ as a quadratic polynomial which is justified by the fact that the runtime complexity of performing $m$ steps of GMRES is $\mathcal{O}(m^2)$ \cite{Joubert}. To determine $h$ we perform runtime measurements of single iterations of GMRES($m$) for multiple restart parameters $m$. We then calculate the coefficients of $h$ by least-squares fitting a quadratic polynomial to the measurement data.
As our heuristic runtime estimate for a PD-GMRES iteration we use
\begin{equation}
\label{eq:heuristic_function}
	\bar{h}(\tilde{r}) = \sum_{\{j \in \setN\ \!|\ \mathrm{Tol}_{\max} > \|\tilde{r}_j\| > \mathrm{Tol}_{\min}\}} \left(\sum_{k} h\big((\tilde{m}_j)_k\big)\right) \frac{\log\left(\frac{\|\tilde{r}_{j-1}\|}{\|\tilde{r}_j\|}\right)}{\log\left(\frac{\|\tilde{r}_{j_{\max}}\|}{\|\tilde{r}_{j_{\min}}\|}\right)},
\end{equation}
where $\tilde{r} = (\tilde{r}_j)$ is the sequence of the residual vectors at each inner iteration of PD-GMRES (not just at each restart), $\tilde{r}_{j_{\min}}$ and $\tilde{r}_{j_{\max}}$ are the residuals with the smallest and largest norm within $[\mathrm{Tol}_{\min}, \mathrm{Tol}_{\max}]$ respectively and $((\tilde{m}_j)_k)$ are the restart parameters for the specific PD-GMRES iteration up to the inner index $j$, yielding the residual $\tilde{r}_j$ (which means that $\sum_{k} ((\tilde{m}_j)_k) = j$). The $((\tilde{m}_j)_k)$ are determined from the restart parameters $(m_k)$ of the complete PD-GMRES iteration. The function $\bar{h}$ from \eqref{eq:heuristic_function} averages the individual heuristics $\sum_{k} h\big((\tilde{m}_j)_k\big)$ for all residuals $\tilde{r}_j$ between $\tilde{r}_{j_{\max}}$ and $\tilde{r}_{j_{\min}}$. Specifically, we have
\begin{equation*}
    \sum_{\{j \in \setN\ \!|\ \mathrm{Tol}_{\max} > \|\tilde{r}_j\| > \mathrm{Tol}_{\min}\}} \frac{\log\left(\frac{\|\tilde{r}_{j-1}\|}{\|\tilde{r}_j\|}\right)}{\log\left(\frac{\|\tilde{r}_{j_{\max}}\|}{\|\tilde{r}_{j_{\min}}\|}\right)} = 1,
\end{equation*}
from which it follows that the heuristic runtime estimates for all residual norms between $\mathrm{Tol}_{\max}$ and $\mathrm{Tol}_{\min}$ are weighted in proportion to the improvement of the residual norm during the $j$-th step of PD-GMRES.
This ensures that instead of optimizing for one specific tolerance we achieve good performance on average over a range of practical values.

We choose a tolerance range from $\mathrm{Tol}_{\max} = 10^{-3}$ to $\mathrm{Tol}_{\min} = 10^{-9}$ as many numerical applications set their target tolerances between these values.

In subsequent analysis we consider multiple different matrices and therefore determine a different heuristic $h$ for each matrix individually, as the runtime costs naturally differ.

For the remainder of this paper when we refer to runtime, we use the averaged heuristic function $\bar{h}$ instead of actual runtime measurements, for PD-GMRES as well as GMRES($m$).
}

\subsubsection{Parameter space considerations} \label{subsec:parameters}
There are five parameters which can be optimized to improve PD-GMRES performance. The initial restart parameter $m_{\text{init}}$, the minimal restart parameter $m_{\min}$, the increment value $m_{\text{step}}$ by which the restart parameter is increased if the value calculated by the controller is less than $m_{\min}$, and the two parameters $\alpha_p$ and $\alpha_d$ which respectively determine the relative impact of the proportional term $P_j(\alpha_p)$ and the derivative term $D_j(\alpha_d)$ in the PD-controller \eqref{eq:pdcontrol}. As to our knowledge there exists no a-priori indication as to what these parameters should be set to.

As the quadtree method allows for optimization of only two parameters simultaneously, we must decide which parameters are grouped together. Due to the structure of the PD-controller it seems natural to optimize $\alpha_p$ and $\alpha_d$ together, as they both directly appear in the control Algorithm~\ref{alg:pd_controller}. Similarly, $m_{\min}$ and $m_{\text{step}}$ are directly related, as they only occur in the reset condition. This leaves $m_{\text{init}}$ as the final parameter and the one which does not occur within the control procedure. Given these considerations and our observation that $\alpha_p$ and $\alpha_d$ appear to be the most impactful parameters for performance, the following optimization strategy is chosen:

\begin{enumerate}
    \item Make an initial choice for $m_{\min}$ and $m_{\text{step}}$.
    \item Choose a value for $m_{\text{init}}$.
    \item Analyze the $(\alpha_p, \alpha_d)$-parameter space with the quadtree method to find improved parameters for the given values of $m_{\text{init}}, m_{\min}, m_{\text{step}}$.
    \item Using these optimized $\alpha_p$ and $\alpha_d$, analyze the $(m_{\min}, m_{\text{step}})$-parameter space with the quadtree method to find improved parameters.
    \item Repeat steps 3 and 4 for a fixed number of cycles.
    \item Repeat from step 2 for different $m_{\text{init}}$ from a preselected set.
    \item Compare the optimal $m_{\min}, m_{\text{step}}, \alpha_p, \alpha_d$ for the different $m_{\text{init}}$ to determine a full set of optimized parameters.
\end{enumerate}

\revision{Each repetition of steps 3 and 4 likely improves the set of optimized parameters. With each optimization cycle one can expect diminishing returns for the performance gain of the parameters, which makes this optimization procedure a practical approach towards optimized, albeit not optimal, parameters. The limited resolution of the quadtree construction prevents a guarantee that a global minimum in the parameter space is found and therefore the optimization procedure can not be expected to converge to a set of globally optimal parameters. However, as observed in the numerical experiments (cf.~Figure~\ref{fig:optimization_sequence}), in practice it is sufficient to terminate the procedure after a few optimization cycles.}

% \revision{The approach of switching between the $(\alpha_p, \alpha_d)$- and $(m_{\min}, m_{\text{step}})$- parameter space can not be expected to be a fast method for achieving convergence to a set of optimal parameters and this is also not the purpose of this optimization strategy. Switching between these parameter spaces instead provides a practical approach towards optimized, albeit not optimal, parameters. As the limited resolution of the quadtree construction does not guarantee that a minimum in the parameter space is found, a theoretical result for convergence cannot be given.}

%Since $\alpha_p$ and $\alpha_d$ are found to be the most impactful parameters for performance, they are optimized first. For this, an initial choice for $m_{min}$ and $m_{\text{step}}$ is made. The $(\alpha_p, \alpha_d)$ parameter space is then analyzed using the quadtree method for a small selection of different values for $m_{\text{init}}$. Best performing parameters are chosen and the $(m_{min}, m_{\text{step}})$ parameter space is then analyzed. This gives a full set of optimized parameters. It is then possible to further improve the parameters by alternating between the two parameter spaces while reducing the parameter ranges.

As mentioned before and in \cite{CuevasNunez2018} there is no 
a-priori indication for the sign of $\alpha_p$ and $\alpha_d$. Because theoretical justification exists both for choosing $\alpha_p < 0$ and $\alpha_d > 0$ by Proposition~\ref{prop:1} as well as for $\alpha_p > 0$ and $\alpha_d > 0$ by Proposition~\ref{prop:2}, these two quadrants of the $(\alpha_p, \alpha_d)$-parameter space are evaluated separately in the numerical experiments (cf.~Section~\ref{sec:numer}).

\subsubsection{Matrix averaging and quadtree resolution} \label{subsec:matrix_averaging}
Applying the optimization procedure for a single matrix gives optimal parameters for PD-GMRES for that specific matrix. It is also possible to apply the method to a set of matrices to find parameters for which PD-GMRES performs well on average for all matrices in the set or even for additional, similar matrices.

After evaluating the quadtrees for all matrices the results are normalized individually relative to the minimal value found in the quadtree for each matrix respectively. The geometric mean over all matrices is then calculated and the minimum in this mean quadtree is taken as the (on average) optimal parameter set. Taking the geometric mean guarantees that the normalized values are optimal on average and further ensures that all matrices are weighted equally. However, it is also possible to take a biased averaging function if the performance for some matrices in the set is considered more important than the one for others.

There is no a-priori method by which it can be determined how high the resolution of the quadtree needs to be in order to provide meaningful data. Additionally, as each evaluation requires calculating a full PD-GMRES iteration, it is desirable to minimize the amount of evaluations to allow the full optimization procedure to be performed in a reasonable amount of time. To estimate the runtime of a single PD-GMRES iteration in advance, we evaluate PD-GMRES runtimes for a set of points within the parameter space and average them (per matrix). In our optimization procedure the depths and ratios for the individual quadtrees are then set using this runtime estimate, such that the total time required to set up the quadtree for each matrix is constrained to a practical amount. Therefore, matrices for which PD-GMRES has higher runtime are optimized using lower resolution quadtrees, which does carry an increased risk of discarding potentially better performing parameters. An alternative would be to create quadtrees of the same resolution for all matrices, at the cost of an increased runtime for the optimization procedure.

As an additional modification, in order to avoid some resolution problems an additional evaluation is performed for the previously best parameters starting at the first optimization in the $(m_{\min}, m_{\text{step}})$-parameter space. This prevents a situation where the new parameters perform worse than the previous ones. It must be noted that this does not guarantee that the average performance actually increases with each optimization cycle, as the low resolution for some of the less well-performing matrices does produce erroneous data points in the quadtree.

\section{Numerical experiments for optimized PD-GMRES} \label{sec:numer}
Before proceeding with our main numerical experiments, we first showcase the behavior of the PD-GMRES iteration and of the quadtree evaluations using two sample matrices.

We use $b = (1, \dots , 1)^T$ as right hand side for all subsequent evaluations.

\subsection{Examples of PD-GMRES and quadtree behavior}
As sample matrices we choose \textit{steam2} and \textit{pde2961} (cf.~Table~\ref{tab:training_matrices}) that are introduced in Section~\ref{sec:training data}.

\begin{figure}
\begin{subfigure}[t]{0.24\textwidth}
\centering
\caption*{\textit{steam2}}
\includegraphics[width=\textwidth]{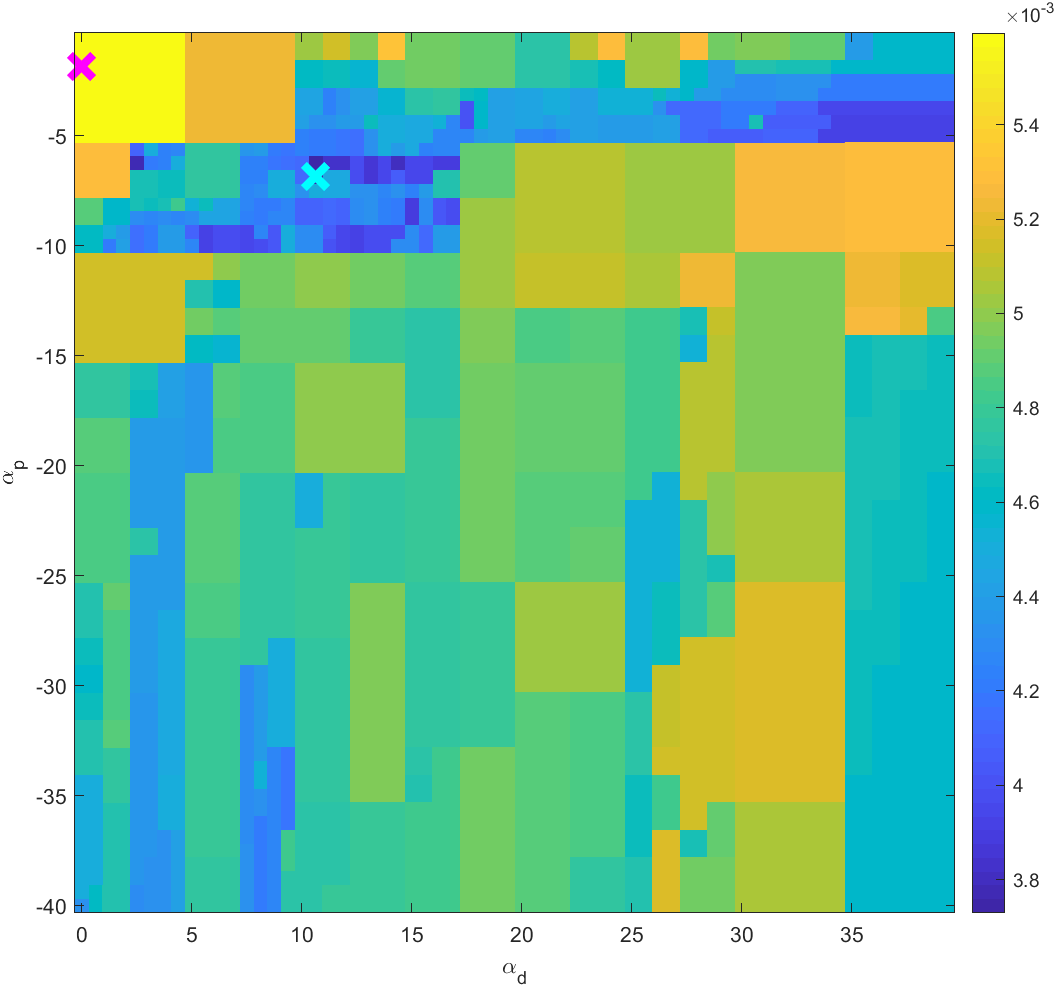}
%\caption*{\textit{steam2} 1a}
\end{subfigure}
\begin{subfigure}[t]{0.24\textwidth}
\centering
\caption*{\textit{pde2961}}
\includegraphics[width=\textwidth]{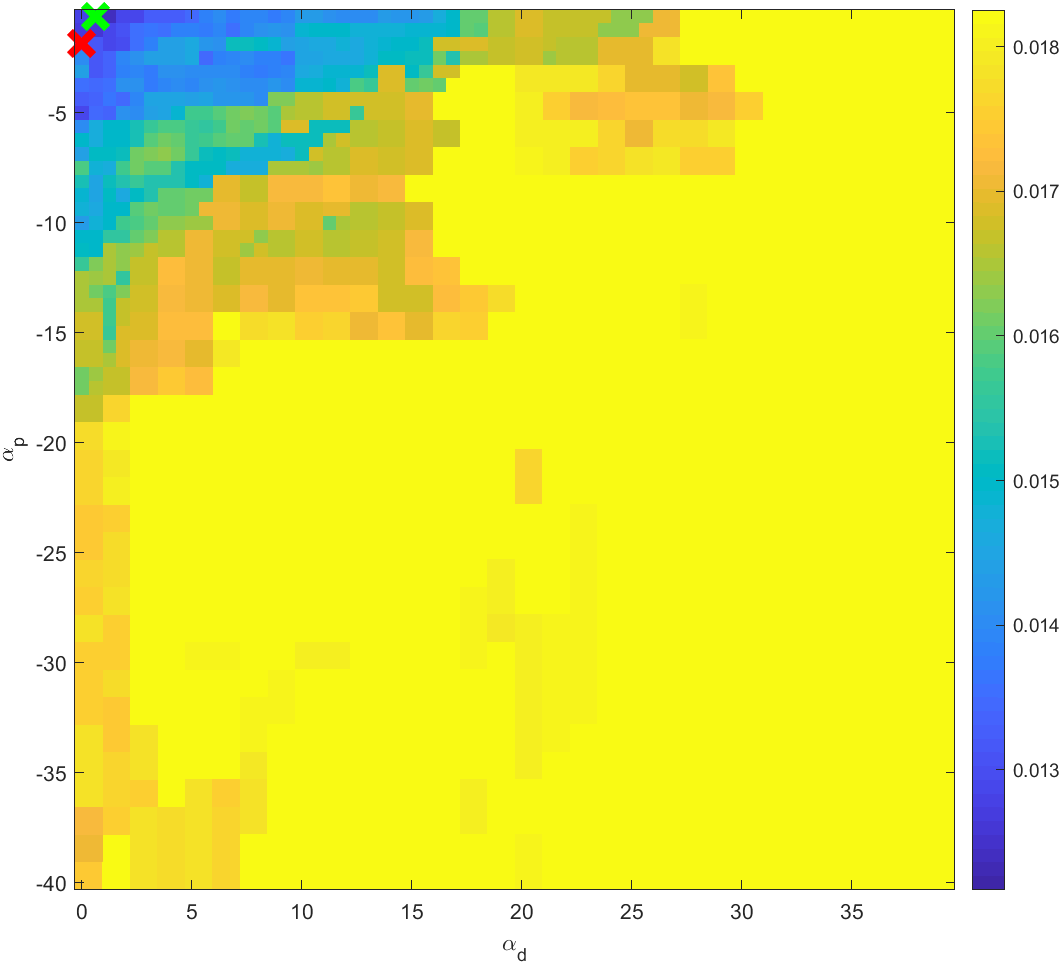}
%\caption*{\textit{pde2961} 1a}
\end{subfigure}
\begin{subfigure}[t]{0.24\textwidth}
\centering
\caption*{Restart parameter}
\includegraphics[width=\textwidth]{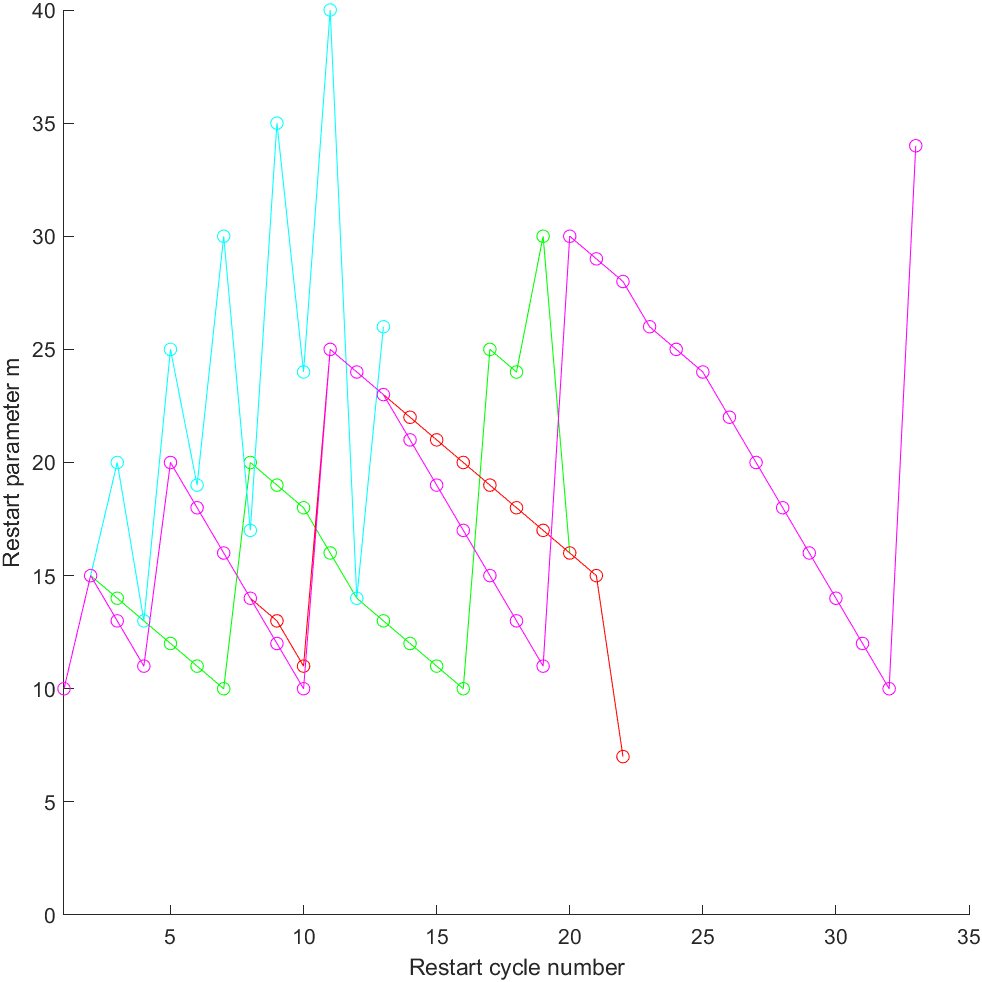}
%\caption*{$m$s 1a}
\end{subfigure}
\begin{subfigure}[t]{0.24\textwidth}
\centering
\caption*{Residual norm}
\includegraphics[width=\textwidth]{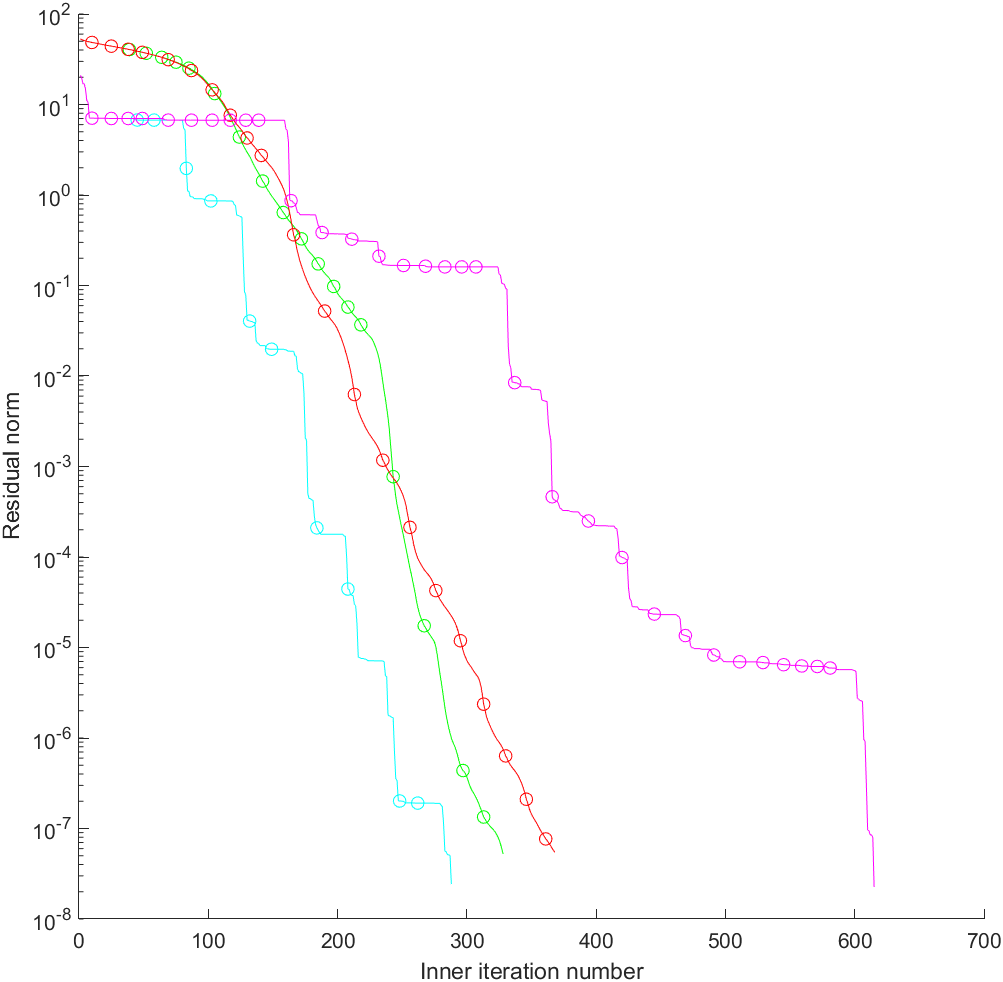}
%\caption*{res 1a}
\end{subfigure}
\centering
%\caption*{$(\alpha_p, \alpha_d)$-space}
\begin{subfigure}[t]{0.24\textwidth}
\centering
\includegraphics[width=\textwidth]{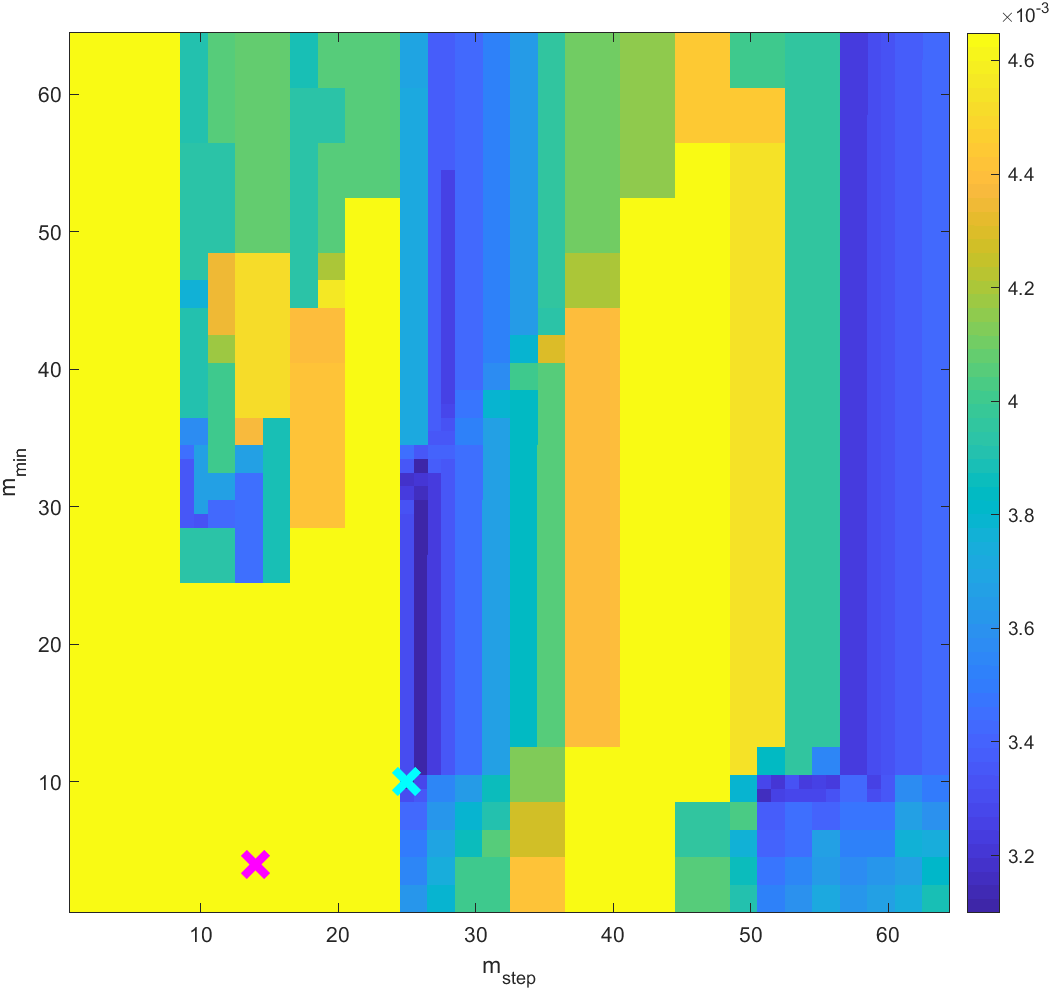}
%\caption*{\textit{steam2} 1b}
\end{subfigure}
\begin{subfigure}[t]{0.24\textwidth}
\centering
\includegraphics[width=\textwidth]{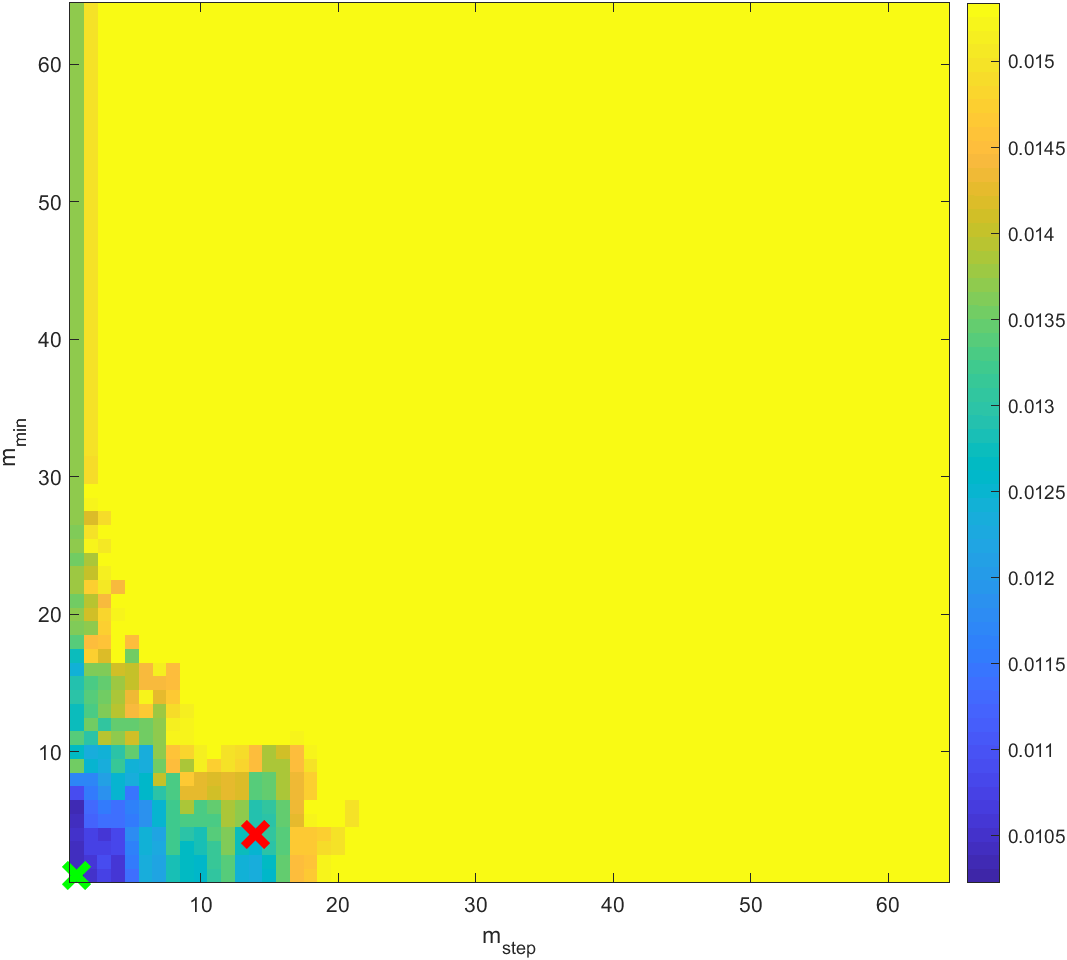}
%\caption*{\textit{pde2961} 1b}
\end{subfigure}
\begin{subfigure}[t]{0.24\textwidth}
\centering
\includegraphics[width=\textwidth]{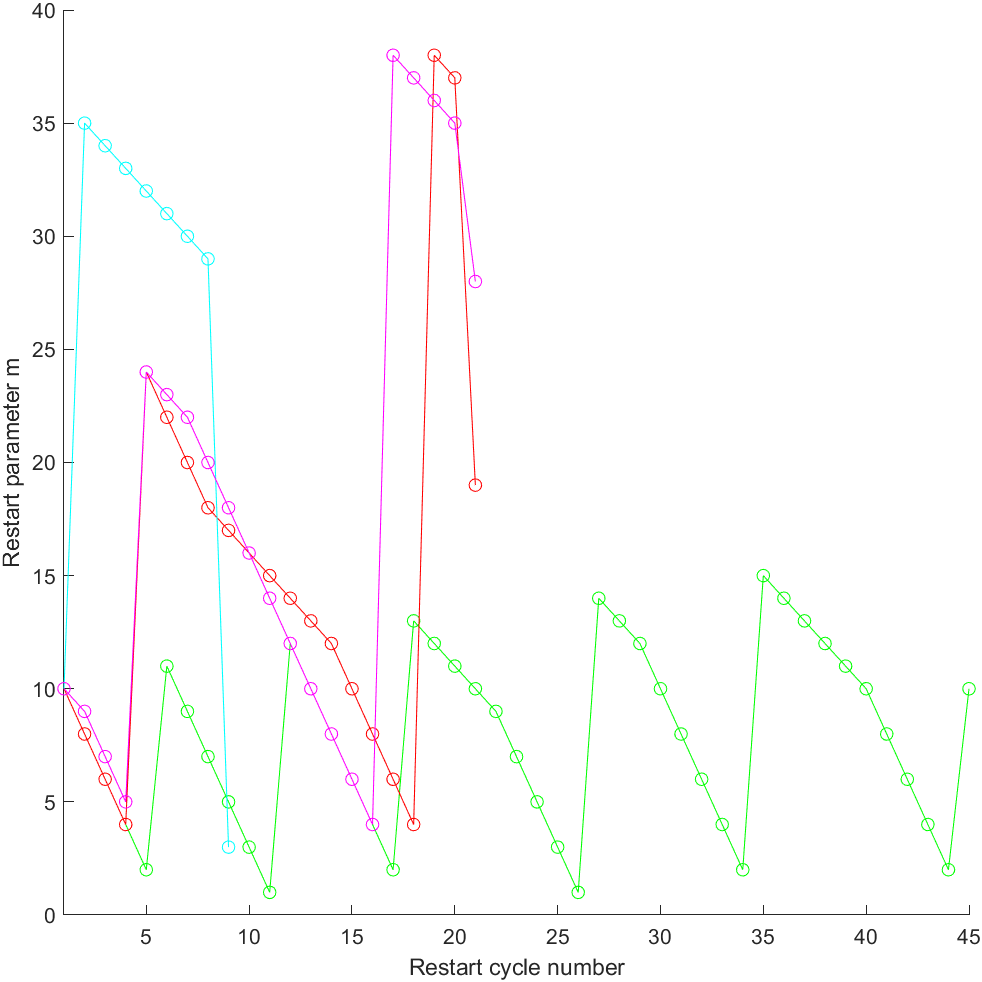}
%\caption*{$m$s 1b}
\end{subfigure}
\begin{subfigure}[t]{0.24\textwidth}
\centering
\includegraphics[width=\textwidth]{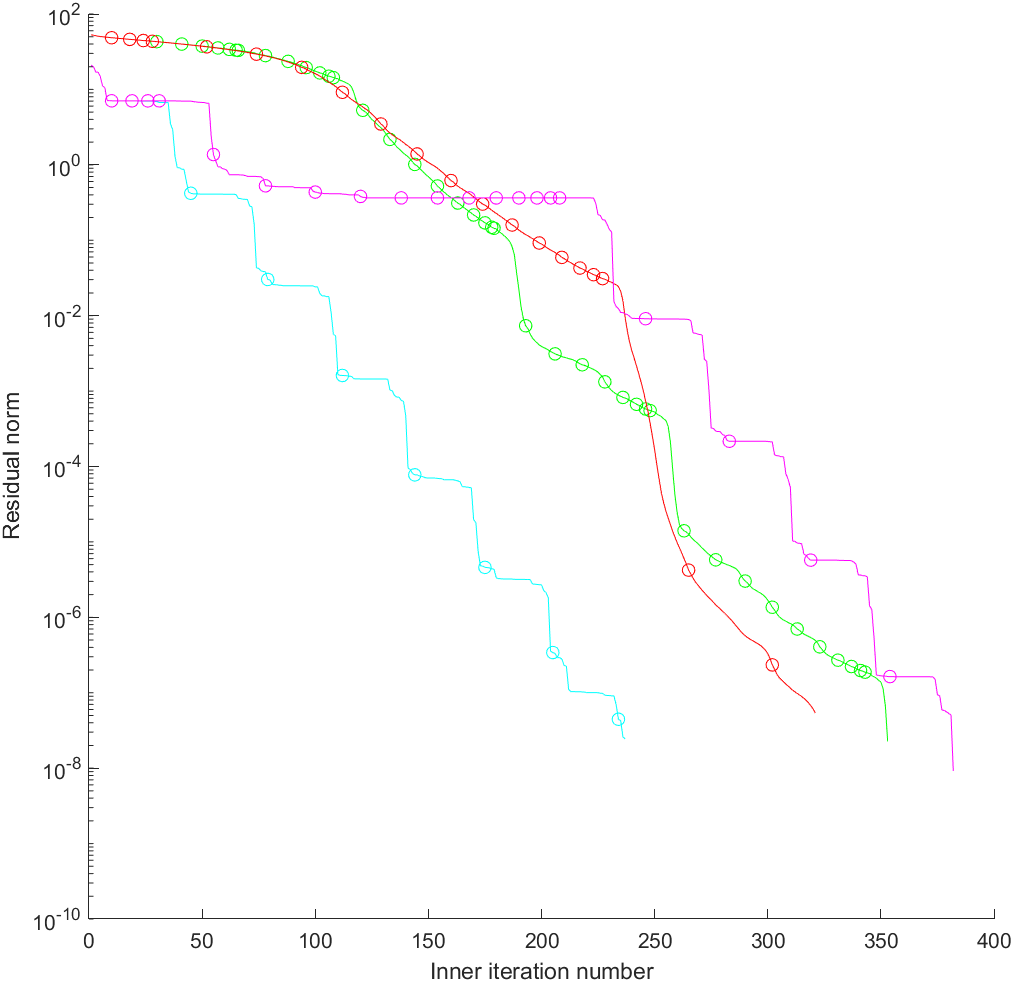}
%\caption*{res 1b}
\end{subfigure}
\centering
%\caption*{$(m_{\min}, m_{\text{step}})$-space}
\caption{Example quadtrees and PD-GMRES iteration behavior for the $(\alpha_p$, $\alpha_d)$-space (top) and the $(m_{\min}, m_{\text{step}})$-space (bottom). \revision{From each quadtree we take two samples for the matrices \textit{steam2} and \textit{pde2961}, taking the local minimum for each (cyan for \textit{steam2} and green for \textit{pde2961}) and the average minimum found accross all matrices (magenta for \textit{steam2} and red for \textit{pde2961}). For these samples the behavior of the restart parameter and residual norm during the respective PD-GMRES iteration are shown.}}
\label{fig:examples}
\end{figure}

In Figure~\ref{fig:examples} the upper row shows the behavior in the $(\alpha_p$, $\alpha_d)$-space, while the lower row shows the $(m_{\min}, m_{\text{step}})$-space. For each, the quadtrees show PD-GMRES performance in parameter space for \textit{steam2} and \textit{pde2961} respectively, while the figures on the right show the behavior of the restart parameter $m$ and the residual norm throughout the PD-GMRES iteration. In the plot of the residual norms, the inner iteration numbers at which restarts occur are marked with circles.

The two matrices were chosen as they represent quite different behavior, which can be taken as an indicator that no averaging optimizer can ever be expected to provide parameters for which all matrices perform well. For each matrix and parameter space, we take two parameter pairs for the evaluations shown on the right. The cyan and green graphs correspond to the local minima in the parameter spaces for the two matrices respectively, while the magenta and red graphs correspond to the parameter pairs which are found to be best when taking the average over a large set of matrices (cf. Section~\ref{sec:training data}).

The quadtrees show the expected behavior of resolving the parameter space at different resolution, with better-performing areas being resolved in greater detail. In the $(\alpha_p$, $\alpha_d)$-space we find that \textit{pde2961} has its best-performing values concentrated near $(0, 0)$, for \textit{steam2} we find well-performing parameters further away from the origin. A similar difference can be found in the $(m_{\min}, m_{\text{step}})$-space, where \textit{pde2961} also shows good performance near $(0, 0)$ while the behavior for \textit{steam2} is very different. We see clear vertical lines of similar performance, which shows that for \textit{steam2} the choice of $m_{\min}$ is less important than the choice of $m_{\text{step}}$, as multiple nearly equally well-performing values for $m_{\min}$ are found for a given $m_{\text{step}}$. It should also be noted that most well-performing values are found for $m_{\min} > 10 = m_{\text{init}}$, which would trigger the reset condition in the very first iteration.

The graphs for the behavior of the restart parameter show a zigzag-pattern which is found for all matrices, as $m$ is slowly decreased by the controller until $m < m_{\min}$ triggers the reset condition. If $\alpha_p$ is small and $\alpha_d$ is large, this occurs more often giving a tighter pattern and if $m_{\text{step}}$ is large then the jumps become larger. The graphs for the residual norm also show a typical pattern, depending on the restart parameter. Here, one can observe sections where the residual norm plateaus as the restart parameter becomes small and after the parameter jumps, the residual norm decreases steeply. It should be further noted that the length of the iteration does not necessarily correlate with its runtime and it is possible for longer iterations (both in terms of number of restart cycles and total inner iterations) to perform better than shorter ones. Specifically, this occurs when a longer iteration features smaller restart parameters, making each individual GMRES call cheaper and resulting in better runtime overall.

\subsection{Applying the optimization approach to training matrices} \label{sec:training data}
We choose a set of training matrices for which we determine optimal PD-GMRES parameters with our opimization procedure. The resulting parameters provide superior runtime compared to GMRES and to PD-GMRES using parameters selected in \cite{CuevasNunez2018}. For the set of training matrices, we initially choose the same set of $23$ matrices from \cite{MatrixDatabase} which was chosen in \cite{Cabral2020}. Since convergence was not achieved for any GMRES or PD-GMRES within a practical amount of time for three matrices (\textit{ex40}, \textit{sherman3}, \textit{wang4}), these three matrices are excluded from the set and the remaining $20$ matrices are taken as training set, see Table~\ref{tab:training_matrices}. Additional considerations for the excluded matrices can be found in Section~\ref{subsec:extend}.

For these training matrices the optimization procedure is performed as described in Section~\ref{sec:optim} with an initial parameter range of $\alpha_p \in [-40, 0]$, $\alpha_d \in [0, 40]$, $m_{\min} \in [1, 65]$, and $m_{\text{step}} \in [1, 65]$. For the first optimization cycle the (arbitrary) choice of $m_{\min} = 10$, $m_{\text{step}} = 5$ is made. A total of three optimization cycles is performed, with the parameter ranges being halved for each repeated evaluation. The choice is made not to increase the level of detail in the quadtrees between optimization cycles. For the $(m_{\min}, m_{\text{step}})$-space this would of course not be possible anyways, and for the $(\alpha_p$, $\alpha_d)$-space one can observe minor improvements by allowing smaller differences in value but those tend towards overfitting for the specific matrices. Furthermore, not increasing the level of detail results in smaller quadtrees, so fewer PD-GMRES evaluations having to be made, which reduces the runtime of the full optimization procedure.

The optimization procedure is performed for $m_{\text{init}} = 10, 20, 30$ and the optimal parameters are found to be $\alpha_p = -0.625, \alpha_d = 4.375, m_{\min} = 3, m_{\text{step}} = 10, m_{\text{init}} = 10$. The PD-GMRES using these parameters will from here on be referred to as the optimized PD-GMRES.

\begin{table}[t]
    \centering
    \begin{tabular}{|c|c|c|c|c|}
    \hline
        Name & n & Non-zeros & Condition & Application area \\ \hline
        add20 &  $2395$ & $13151$ & $1.76e+04$ & circuit simulation problem \\ \hline
		cavity05 &  $1182$ & $32632$ & $9.18e+05$ & computational fluid dynamics problem sequence \\ \hline
		cavity10 &  $2597$ & $76171$ & $4.46e+06$ & computational fluid dynamics problem sequence \\ \hline
		cdde1 &  $961$ & $4681$ & $4.08e+03$ & computational fluid dynamics problem sequence \\ \hline
		circuit\_2 &  $4510$ & $21199$ & $7.12e+06$ & circuit simulation problem \\ \hline
\rowcolor{lightgray}ex40 &  $7740$ & $456188$ & $1.20e+07$ & computational fluid dynamics problem \\ \hline
		fpga\_trans\_01 &  $1220$ & $7382$ & $2.88e+04$ & circuit simulation problem sequence \\ \hline
		memplus &  $17758$ & $99147$ & $2.67e+05$ & circuit simulation problem \\ \hline
		orsirr\_1 &  $1030$ & $6858$ & $1.67e+05$ & computational fluid dynamics problem \\ \hline
		orsreg\_1 &  $2205$ & $14133$ & $1.54e+04$ & computational fluid dynamics problem \\ \hline
		pde2961 &  $2961$ & $14585$ & $9.49e+02$ & 2D/3D problem \\ \hline
		raefsky1 &  $3242$ & $293409$ & $3.16e+04$ & computational fluid dynamics problem sequence \\ \hline
		raefsky2 &  $3242$ & $293551$ & $1.08e+04$ & subsequent computational fluid dynamics problem \\ \hline
		rdb2048 &  $2048$ & $12032$ & $1.96e+03$ & computational fluid dynamics problem \\ \hline
		sherman1 &  $1000$ & $3750$ & $2.26e+04$ & computational fluid dynamics problem \\ \hline
\rowcolor{lightgray}sherman3 &  $5005$ & $20033$ & $6.90e+16$ & computational fluid dynamics problem \\ \hline
		sherman4 &  $1104$ & $3786$ & $7.16e+03$ & computational fluid dynamics problem \\ \hline
		sherman5 &  $3312$ & $20793$ & $3.90e+05$ & computational fluid dynamics problem \\ \hline
		steam2 &  $600$ & $5660$ & $3.55e+06$ & computational fluid dynamics problem \\ \hline
		wang2 &  $2903$ & $19093$ & $3.18e+04$ & subsequent semiconductor device problem \\ \hline
\rowcolor{lightgray}wang4 &  $26068$ & $177196$ & $4.91e+04$ & semiconductor device problem \\ \hline
		watt\_1 &  $1856$ & $11360$ & $5.38e+09$ & computational fluid dynamics problem \\ \hline
		young3c &  $841$ & $3988$ & $1.15e+04$ & acoustics problem \\ \hline
    \end{tabular}
    \caption{Name, matrix dimension, number of non-zeros, condition number (1-norm condition estimate) and application area, for the set of training matrices (same as in \cite{Cabral2020}). We exclude matrices marked in gray from further evaluation due to poor performance.}
    \label{tab:training_matrices}
\end{table}

\begin{figure}
\begin{subfigure}{0.33\textwidth}
\centering
\caption*{First optimization cycle}
\includegraphics[width=\textwidth]{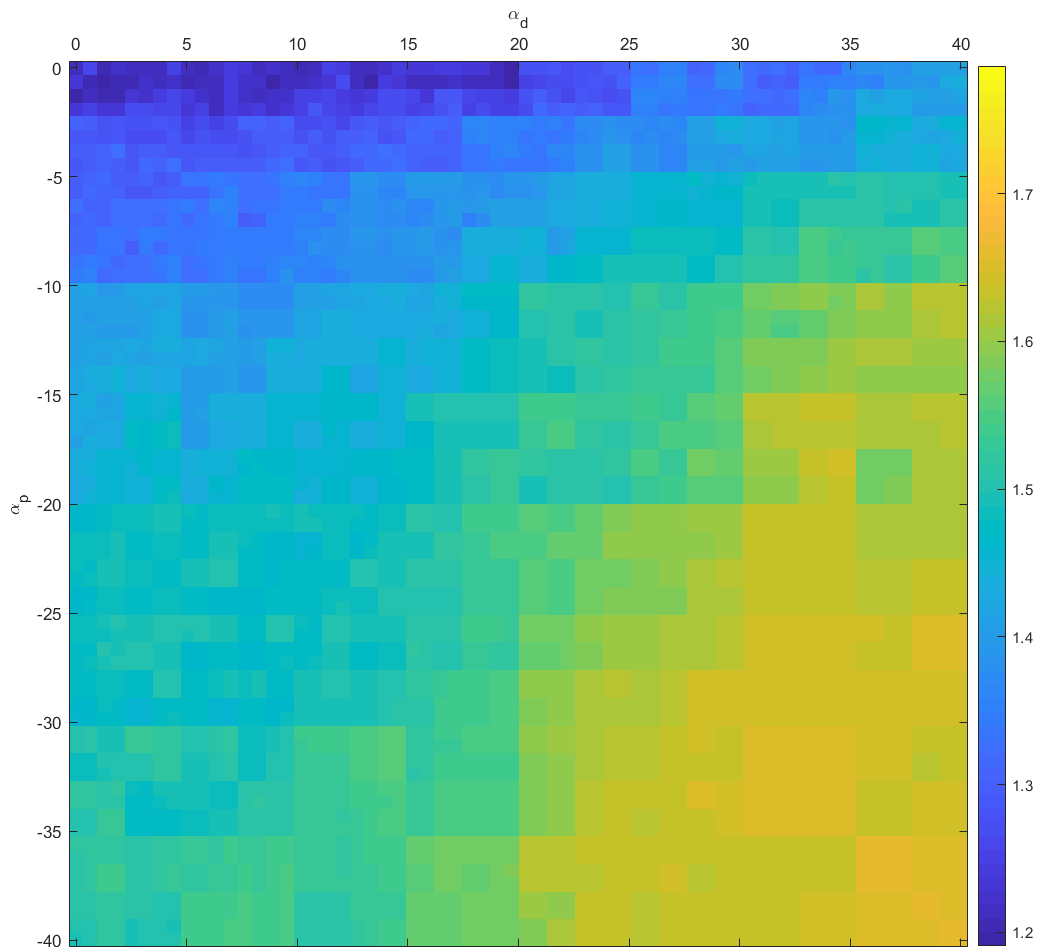}
\end{subfigure}
\begin{subfigure}{0.33\textwidth}
\centering
\caption*{Second optimization cycle}
\includegraphics[width=\textwidth]{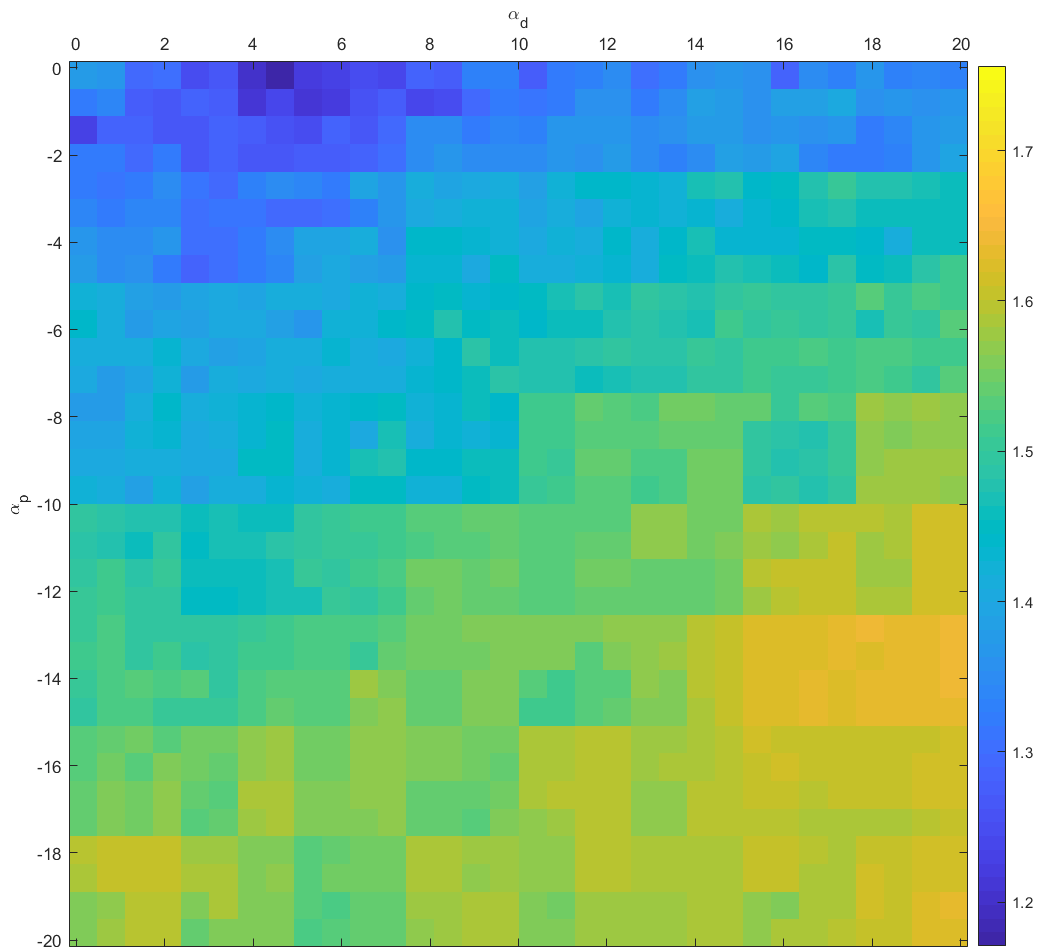}
\end{subfigure}
\begin{subfigure}{0.33\textwidth}
\centering
\caption*{Third optimization cycle}
\includegraphics[width=\textwidth]{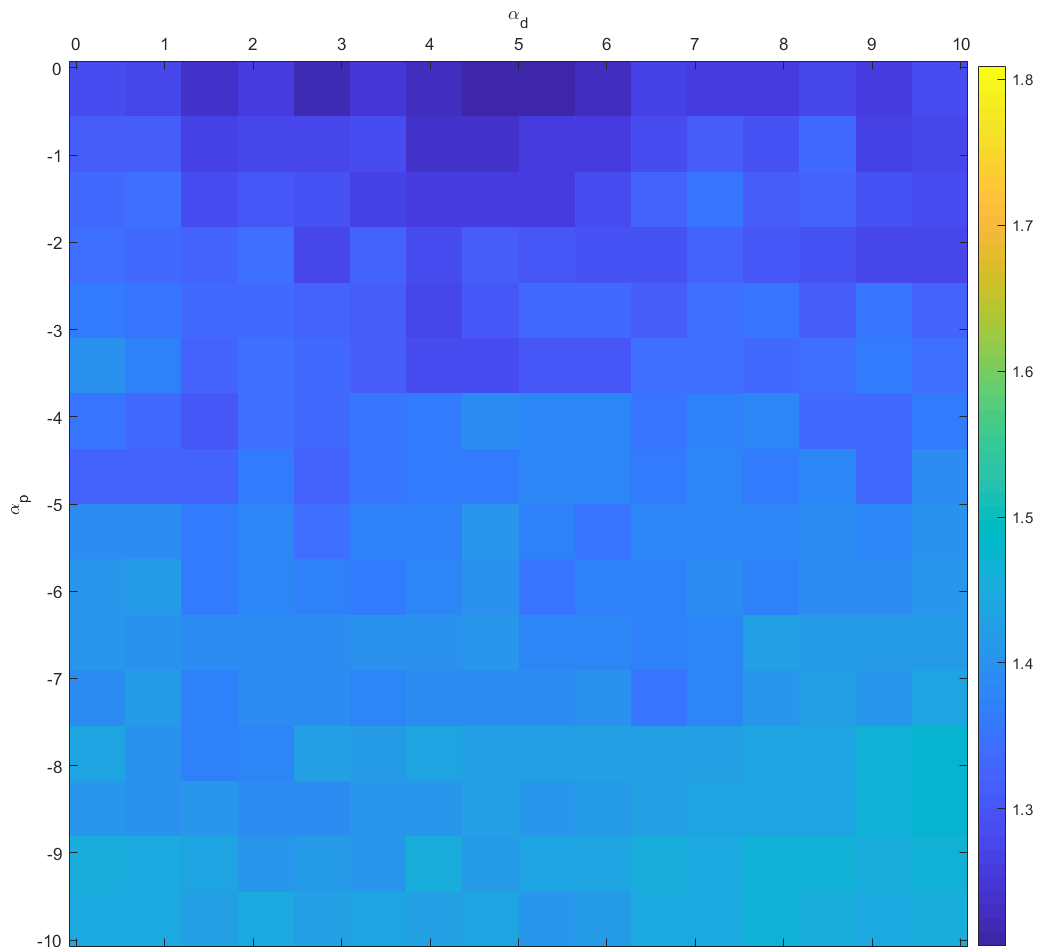}
\end{subfigure}
\begin{subfigure}{0.33\textwidth}
\vspace*{0.5cm}
\centering
\includegraphics[width=\textwidth]{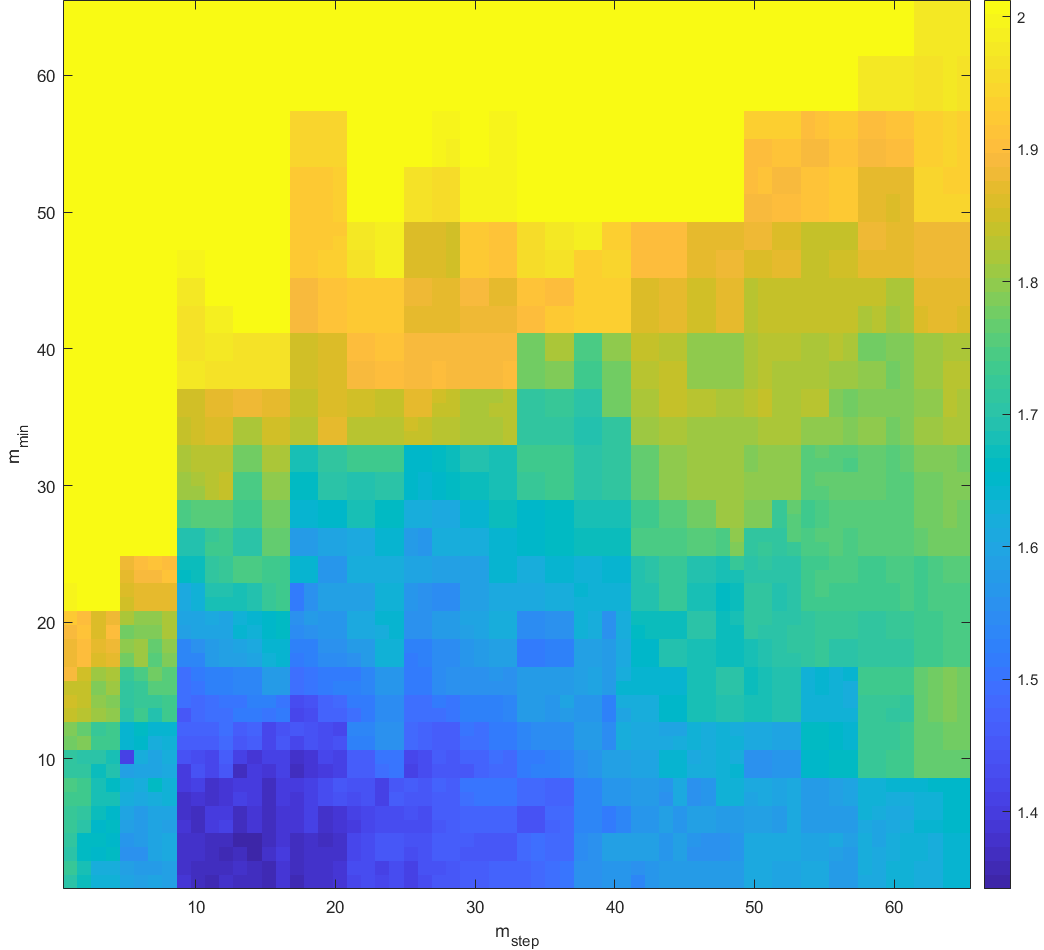}
\end{subfigure}
\begin{subfigure}{0.33\textwidth}
\centering
\includegraphics[width=\textwidth]{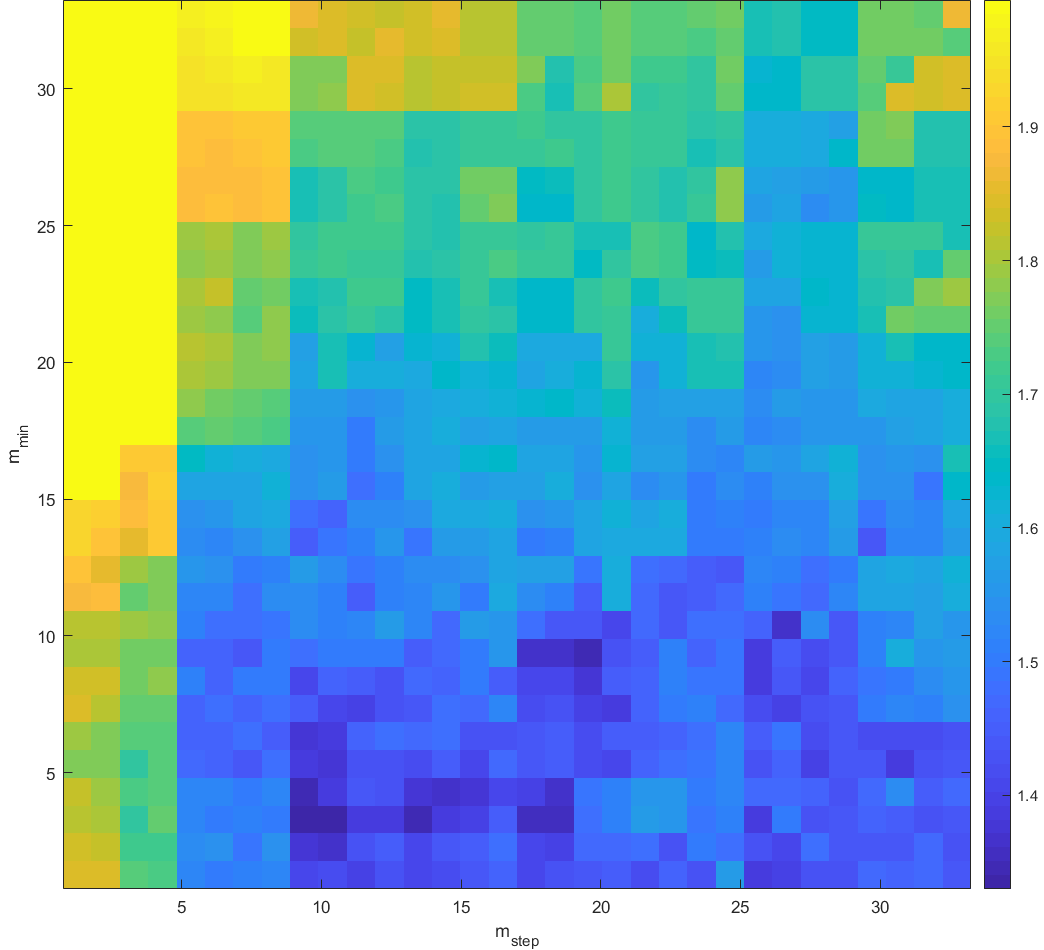}
\end{subfigure}
\begin{subfigure}{0.33\textwidth}
\centering
\includegraphics[width=\textwidth]{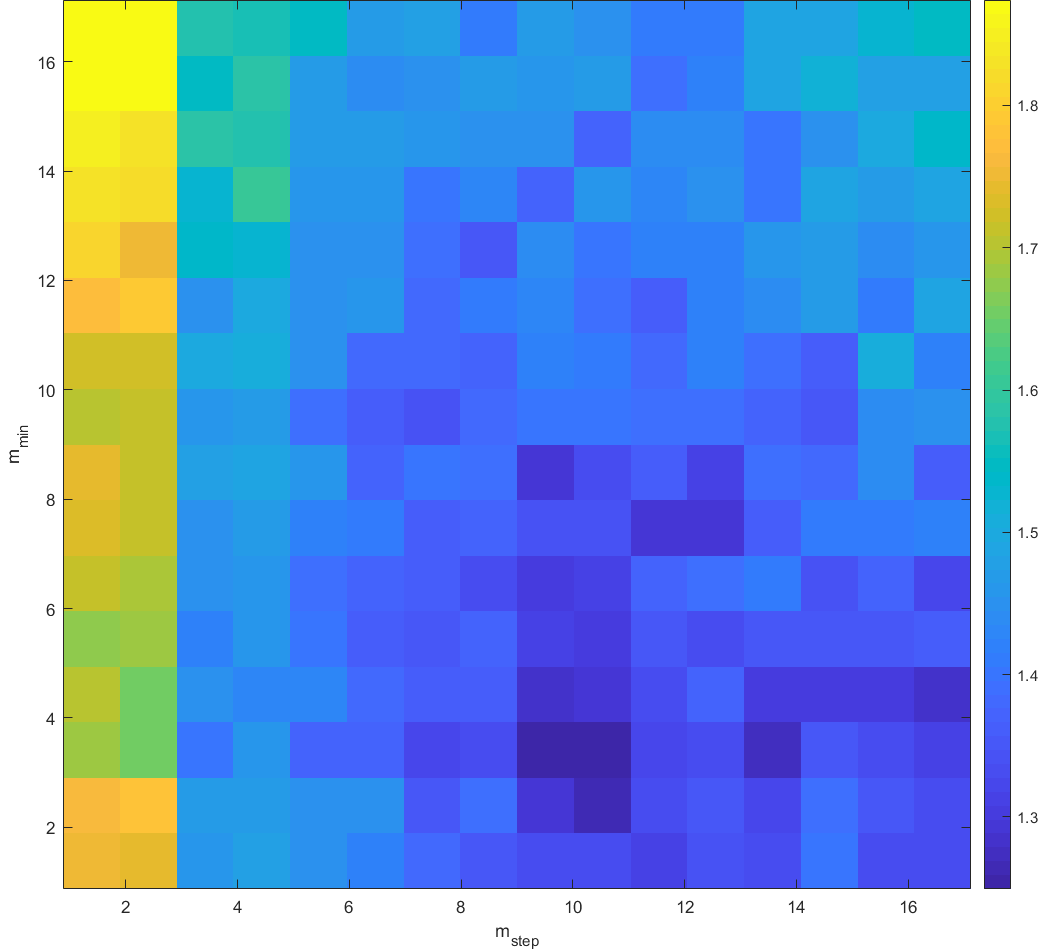}
\end{subfigure}
\caption{Average quadtrees calculated by the optimization sequence for the training matrices using $m_{\text{init}} = 10$.\\
Final parameters for the optimized PD-GMRES are $m_{\text{init}} = 10, m_{\min} = 3, m_{\text{step}} = 10, \alpha_p = -0.625, \alpha_d = 4.375$. The top row depicts the $(\alpha_p,\alpha_d)$-space, and the  bottom row the $(m_{\min},m_{\text{step}})$-space. \revision{The geometric means for the average matrix performance (cf.~Section \ref{subsec:matrix_averaging}) after each optimization step are: $1,\ 0.9270,\ 0.9060,\ 0.8957,\ 0.8957,\ 0.8953$. These values are normalized around the average performance after the first optimization step which already improves performance significantly over GMRES($m$). We observe that the gains in matrix performance decrease with further optimization steps and therefore we terminate the optimization procedure after three cycles.}} 
\label{fig:optimization_sequence}
\end{figure}

Figure~\ref{fig:optimization_sequence} shows the six averaged quadtrees for the optimization sequence for $m_{\text{init}} = 10$. For each optimization cycle the top quadtree shows the $(\alpha_p, \alpha_d)$-space and the bottom quadtree the $(m_{\min}, m_{\text{step}})$-space. Some resolution effects can be observed in particular for the $(m_{\min}, m_{\text{step}})$-space, where lower-resolution evaluations for certain matrices result in a vertical strip on the left side of the parameter space. Within this strip the additionally computed value at $m_{\min} = 10$, $m_{\text{step}} = 5$ can also be observed in the first cycle, which does indicate that additional evaluations in this area could find better performance than the current resolution suggests. Color scaling is set for each average quadtree individually with blue corresponding to the minimal value found and yellow corresponding to $1.5$ times the minimal value.

\begin{table}[H]
	\addtolength{\tabcolsep}{-0.2em}
    \begin{tabular}{c|c|cccccccccccccccc}
         \multicolumn{2}{c|}{\multirow{2}{*}{}} & \multicolumn{16}{c}{$m_{\text{step}}$} \\ \cline{3-18}
          \multicolumn{2}{c|}{} & 1 & 2 & 3 & 4 & 5 & 6 & 7 & 8 & 9 & 10 & 11 & 12 & 13 & 14 & 15 & 16\\ \hline
        \multirow{16}{*}{$m_{\min}$} & 1 & 1.75 & 1.75 & 1.46 & 1.48 & 1.45 & 1.42 & 1.38 & 1.35 & 1.34 & 1.33 & 1.32 & 1.34 & 1.33 & 1.40 & 1.33 & 1.33 \\
& 2 & 1.76 & 1.78 & 1.47 & 1.47 & 1.45 & 1.45 & 1.35 & 1.39 & 1.29 & 1.27 & 1.33 & 1.36 & 1.32 & 1.39 & 1.35 & 1.34 \\
& 3 & 1.69 & 1.65 & 1.40 & 1.46 & 1.37 & 1.38 & 1.32 & 1.33 & 1.26 & \textbf{1.25} & 1.32 & 1.33 & 1.27 & 1.35 & 1.33 & 1.31 \\
& 4 & 1.70 & 1.65 & 1.45 & 1.43 & 1.43 & 1.38 & 1.36 & 1.36 & 1.28 & 1.29 & 1.33 & 1.37 & 1.31 & 1.30 & 1.31 & 1.29 \\
& 5 & 1.68 & 1.68 & 1.42 & 1.46 & 1.40 & 1.36 & 1.35 & 1.37 & 1.31 & 1.31 & 1.35 & 1.34 & 1.35 & 1.35 & 1.35 & 1.37 \\
& 6 & 1.71 & 1.69 & 1.45 & 1.46 & 1.39 & 1.37 & 1.36 & 1.34 & 1.30 & 1.31 & 1.37 & 1.39 & 1.41 & 1.34 & 1.37 & 1.32 \\
& 7 & 1.73 & 1.71 & 1.45 & 1.47 & 1.42 & 1.41 & 1.36 & 1.37 & 1.34 & 1.34 & 1.29 & 1.30 & 1.36 & 1.41 & 1.41 & 1.42 \\
& 8 & 1.74 & 1.71 & 1.48 & 1.49 & 1.46 & 1.37 & 1.40 & 1.40 & 1.29 & 1.33 & 1.36 & 1.31 & 1.39 & 1.38 & 1.44 & 1.36 \\
& 9 & 1.71 & 1.71 & 1.46 & 1.47 & 1.39 & 1.37 & 1.34 & 1.38 & 1.41 & 1.40 & 1.39 & 1.39 & 1.37 & 1.35 & 1.44 & 1.45 \\
& 10 & 1.72 & 1.72 & 1.50 & 1.50 & 1.45 & 1.38 & 1.38 & 1.37 & 1.42 & 1.41 & 1.39 & 1.42 & 1.39 & 1.36 & 1.50 & 1.42 \\
& 11 & 1.77 & 1.79 & 1.45 & 1.49 & 1.45 & 1.46 & 1.39 & 1.41 & 1.43 & 1.39 & 1.36 & 1.42 & 1.44 & 1.47 & 1.41 & 1.49 \\
& 12 & 1.81 & 1.76 & 1.54 & 1.52 & 1.45 & 1.45 & 1.39 & 1.35 & 1.44 & 1.40 & 1.42 & 1.42 & 1.46 & 1.47 & 1.44 & 1.46 \\
& 13 & 1.83 & 1.82 & 1.53 & 1.60 & 1.46 & 1.46 & 1.40 & 1.43 & 1.37 & 1.46 & 1.43 & 1.44 & 1.40 & 1.49 & 1.47 & 1.49 \\
& 14 & 1.85 & 1.83 & 1.58 & 1.58 & 1.47 & 1.47 & 1.46 & 1.45 & 1.45 & 1.37 & 1.44 & 1.44 & 1.40 & 1.45 & 1.50 & 1.54 \\
& 15 & 1.92 & 1.92 & 1.54 & 1.58 & 1.47 & 1.44 & 1.45 & 1.47 & 1.46 & 1.47 & 1.39 & 1.42 & 1.49 & 1.52 & 1.48 & 1.48 \\
& 16 & 1.92 & 1.93 & 1.57 & 1.56 & 1.54 & 1.47 & 1.48 & 1.41 & 1.46 & 1.45 & 1.41 & 1.41 & 1.49 & 1.49 & 1.53 & 1.55 \\ 
    \end{tabular}
    \caption{PD-GMRES runtime in the $(m_{\min}, m_{\text{step}})$-space for $m_{\text{init}} = 10, \alpha_p = -0.625, \alpha_d = 4.375$, corresponding to the last subfigure in Figure~\ref{fig:optimization_sequence}. Values are averaged across all training matrices for which convergence is achieved and then normalized against the minimal value for each individual matrix. Best performance is found for $m_{\min} = 3, m_{\text{step}} = 10$.}
    \label{tab:mmin_mstep_optimization}
\end{table}

Table~\ref{tab:mmin_mstep_optimization} shows the runtime values from the last subfigure in Figure~\ref{fig:optimization_sequence}, with a minimum of $1.25$ found for $m_{\min} = 3, m_{\text{step}} = 10$. A value of $1.0$ would correspond to a parameter pair for which each matrix has minimal runtime. The optimal parameter pair being found at a value of $1.25$ means that on average one can expect performance to be $25$ percent worse for the generic choice of $(m_{\min}, m_{\text{step}})$ compared to the optimal parameters one would find for each matrix individually.

\begin{figure}
\begin{subfigure}{\textwidth}
\centering
\includegraphics[width=\textwidth]{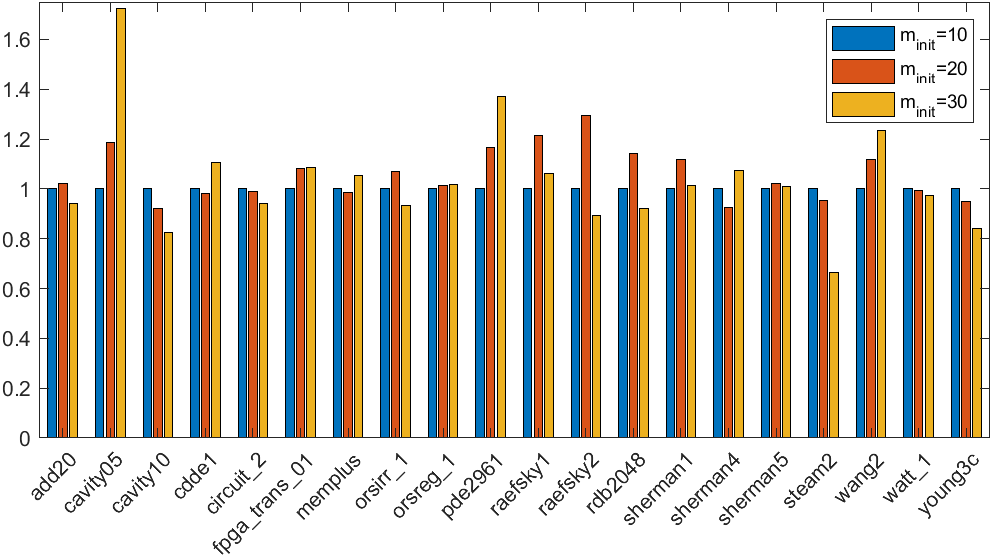}
\end{subfigure}
\caption{Comparison of PD-GMRES performance for training matrices, optimized using $m_{\text{init}} = 10, 20, 30$:\\
$m_{\text{init}} = 10, m_{\min} = 3, m_{\text{step}} = 10, \alpha_p = -0.625, \alpha_d = 4.375$, geometric mean = $1$\\
$m_{\text{init}} = 20, m_{\min} = 1, m_{\text{step}} = 11, \alpha_p = -0.625, \alpha_d = 3.125$, geometric mean = $1.0522$\\
$m_{\text{init}} = 30, m_{\min} = 5, m_{\text{step}} = 5, \alpha_p = -0.625, \alpha_d = 5.625$, geometric mean = $1.0142$\\}
%\caption{Comparing optimized PD-GMRES for training matrices for different $m_{\text{init}}$:\\
%Geometric mean = $1$ for $m_{\text{init}} = 10$, Geometric mean =  $1.0522$ for $m_{\text{init}} = 20$, Geometric mean = $1.0142$ for $m_{\text{init}} = 30$}
\label{fig:comparing_minit}
\begin{subfigure}{\textwidth}
\includegraphics[width=\textwidth]{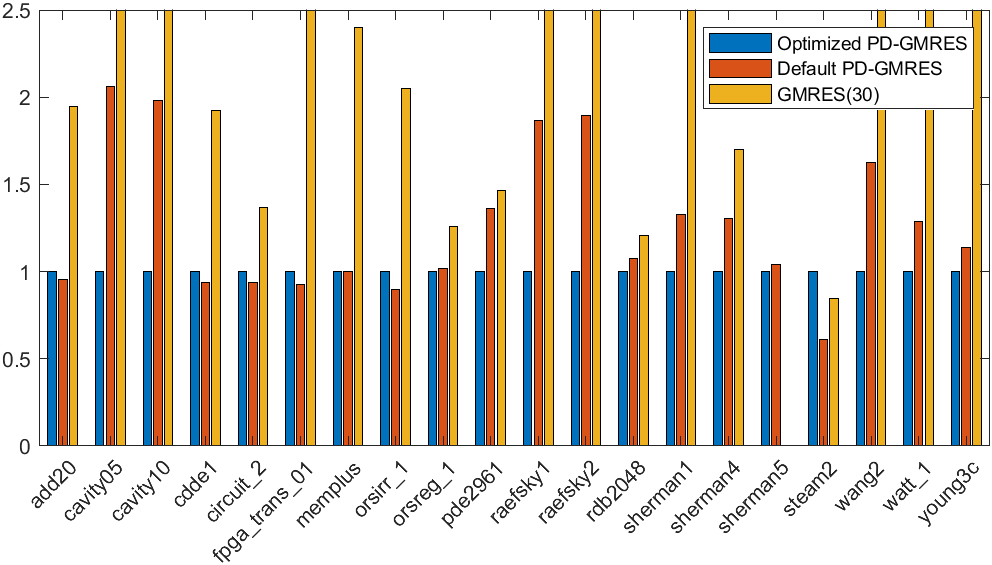}
\end{subfigure}
\caption{Comparison of optimized PD-GMRES performance against default PD-GMRES and GMRES(30):\\
Optimized PD-GMRES uses the parameters found for $m_{\text{init}} = 10$ (cf.~Figure~\ref{fig:comparing_minit}).\\
Default PD-GMRES uses $m_{\text{init}} = 30, m_{\min} = 1, m_{\text{step}} = 3, \alpha_p = -3, \alpha_d = 9$.
% \\
% The geometric means are $1$ for optimized PD-GMRES, $1.2020$ for default PD-GMRES, $2.7596$ for GMRES(30).\\
% We note, that GMRES(30) does not converge for \textit{sherman5} (iteration stagnates) and therefore it was excluded when calculating the geometric mean for GMRES(30).
}
\label{fig:training_nunez_gmres_comp}
\end{figure}

Figure~\ref{fig:comparing_minit} shows the PD-GMRES performance for the training matrices with three values of $m_{\text{init}}$ ($m_{\text{init}} = 10, 20, 30$). It can be observed that for all three choices of $m_{\text{init}}$ the optimized PD-GMRES performs similarly well for most matrices, with $m_{\text{init}} = 10$ being the optimal choice on average. The other PD-GMRES parameters are also similar for all three values of $m_{\text{init}}$, indicating that the choice of $m_{\text{init}}$ is not decisive for the average performance of PD-GMRES for the training matrices.

Figure~\ref{fig:training_nunez_gmres_comp} shows the normalized performance of the optimized PD-GMRES compared to the default PD-GMRES using the parameters chosen in \cite{CuevasNunez2018} ($m_{\text{init}} = 30, m_{\min} = 1, m_{\text{step}} = 3, \alpha_p = -3, \alpha_d = 9$) and to GMRES($30$). The geometric means are $1$ for optimized PD-GMRES, $1.2020$ for default PD-GMRES, and $2.7596$ for GMRES(30). We note, that GMRES(30) does not converge for \textit{sherman5} (iteration stagnates) and therefore it was excluded when calculating the corresponding geometric mean. The optimized PD-GMRES thus performs better than both other algorithms on average and there exists only one matrix (\textit{steam2}) for which the optimized PD-GMRES performs notably worse.

However, when comparing PD-GMRES to GMRES($m$), taking values other than $m = 30$ is useful to measure the quality of the optimized PD-GMRES. Note that GMRES($30$) is not the best performing GMRES($m$) for all matrices. This means that outperforming GMRES($30$) but being outperformed by some other GMRES($m$) would be an indicator of poor quality of the optimized PD-GMRES.

\begin{table}
\addtolength{\tabcolsep}{-0.1em}
\begin{tabular}{c|cccccc}
name & PD-GMRES & GMRES(10) & GMRES(20) & GMRES(30) & GMRES(50) & GMRES(100)\\
\hline
add20 &  \textbf{0.0297} & 0.1221  & 0.0897  & 0.0579 & 0.0463 & 0.0555  \\
cavity05 &  \textbf{0.0518} & 4.0057  & 0.8785  & 0.3778 & 1.1493 & 0.4035  \\
cavity10 &  \textbf{0.2586} & 52.822  & 29.639 & 6.7228 & 5.0075 & 11.419 \\
cdde1 &  0.0084 & 0.3228  & 0.0289  & 0.0162 & 0.0083 & \textbf{0.0073}  \\
circuit\_2 &  0.1559 & 1.0721  & 0.5099  & 0.2128 & \textbf{0.1420} & 0.1449  \\
fpga\_trans\_01 &  \textbf{0.0357} & 0.6626  & 0.2559  & 0.1336 & 0.0593 & 0.0526  \\
memplus &  \textbf{1.2670} & 6.0138  & 4.4135  & 3.0401 & 3.8703 & 3.5727  \\
orsirr\_1 &  \textbf{0.0472} & 0.1872  & 0.1823  & 0.0968 & 0.0633 & 0.0677  \\
orsreg\_1 &  \textbf{0.0115} & 0.0154  & 0.0151  & 0.0145 & 0.0145 & 0.0195  \\
pde2961 &  \textbf{0.0138} & 0.0182  & 0.0160  & 0.0203 & 0.0334 & 0.0612  \\
raefsky1 &  \textbf{0.1493} & 3.0280  & 1.7792  & 0.7574 & 0.7875 & 0.8685  \\
raefsky2 &  \textbf{0.3096} & 1.3139  & 0.9220  & 1.0826 & 1.2424 & 1.0277  \\
rdb2048 &  0.0132 & 0.0161  & 0.0147  & 0.0160 & 0.0123 & \textbf{0.0118}  \\
sherman1 &  \textbf{0.0178} & 0.1650  & 0.0918  & 0.0706 & 0.0648 & 0.0612  \\
sherman4 &  0.0069 & 0.0142  & 0.0128  & 0.0117 & 0.0093 & \textbf{0.0062}  \\
sherman5 &  \textbf{1.4181} & -     & -     & -    & 5.2856 & 3.9816  \\
steam2 &  0.0065 & -     & 15.593 & 0.0055 & 0.0057 & \textbf{0.0046}  \\
wang2 &  \textbf{0.0342} & 14.390 & 0.1315  & 0.1296 & 0.1310 & 0.0852  \\
watt\_1 &  \textbf{0.0231} & 0.2555  & 0.1487  & 0.1215 & 0.0936 & 0.0496  \\
young3c &  \textbf{0.0468} & 0.3404  & 0.1816  & 0.1381 & 0.1174 & 0.1190 
\end{tabular}
\caption{Runtimes for optimized PD-GMRES compared to various GMRES($m$) for the training matrices. For each problem the best performing algorithm is marked in bold. Missing entries indicate failure to converge within 240 seconds.}
\label{tab:training_gmres_comp}
\end{table}

\begin{table}
\centering
\begin{tabular}{|cccccc|}
	\hline
	PD-GMRES & GMRES(10) & GMRES(20) & GMRES(30) & GMRES(50) & GMRES(100)\\
	\hline
	$1$ & $9.8105$ & $6.1013$ & $2.7596$ & $2.5647$ & $2.3776$\\
	\hline
\end{tabular}
\caption{Geometric means of normalized data from Table~\ref{tab:training_gmres_comp} (excluding non-converging matrices column-wise).}
\label{tab:training_gmres_geomeans}
\end{table}

Table~\ref{tab:training_gmres_comp} compares the runtimes for the optimized PD-GMRES, as previously determined, to GMRES($m$) for several different values of $m$. Missing entries indicate that GMRES($m$) did not achieve convergence in a reasonable amount of time ($240$ seconds). This either means that the iteration has reached true stagnation, so no amount of additional time could result in convergence, or that the iteration merely converges very slowly. Overall, for the majority of matrices the optimized PD-GMRES performs better than all GMRES, however, there are some matrices which exhibit the behavior that GMRES($100$) actually performs better than PD-GMRES. These matrices shall be referred to as "problematic" from here on.

To further quantify the advantage of PD-GMRES over fixed-restart GMRES, the geometric means of normalized data from Table~\ref{tab:training_gmres_comp} are calculated for the different GMRES($m$) and shown in Table~\ref{tab:training_gmres_geomeans}. For non-converging GMRES the corresponding rows were discarded when calculating the geometric mean. Therefore, the values for $m = 10, 20, 30$ should be considered lower bounds on the geometric means, i.e., PD-GMRES outperforms these GMRES($m$) by an even larger factor.

\subsection{Evaluating performance of optimized PD-GMRES on test matrices}
To test the optimized PD-GMRES parameters which were found by optimizing the PD-GMRES for the training matrices, a new set of matrices is selected from \cite{MatrixDatabase}. The selection process was done in two parts. First, matrices relating to the training matrices were selected manually and additional matrices were chosen randomly. Then, matrices for which no convergence was achieved for any GMRES or PD-GMRES in reasonable time spans were discarded, yielding the test matrix set in Table~\ref{tab:test_matrices}.

\begin{table}[h]
    \centering
    \begin{tabular}{|c|c|c|c|c|}
    \hline
        Name & n & Non-zeros & Condition & Application area \\ \hline
ACTIVSg2000 &  $4000$ & $28505$ & $1.38e+07$ & power network problem \\ \hline
CAG\_mat1916 &  $1916$ & $195985$ & $1.37e+08$ & combinatorial problem \\ \hline
Chem97ZtZ &  $2541$ & $7361$ & $4.63e+02$ & statistical/mathematical problem \\ \hline
add32 &  $4960$ & $19848$ & $2.14e+02$ & circuit simulation problem \\ \hline
adder\_dcop\_01 &  $1813$ & $11156$ & $1.30e+08$ & subsequent circuit simulation problem \\ \hline
bcsstk34 &  $588$ & $21418$ & $5.09e+04$ & structural problem \\ \hline
bfwa398 &  $398$ & $3678$ & $7.58e+03$ & electromagnetics problem \\ \hline
cavity01 &  $317$ & $7280$ & $6.91e+04$ & computational fluid dynamics problem sequence \\ \hline
cdde2 &  $961$ & $4681$ & $1.29e+02$ & subsequent computational fluid dynamics problem \\ \hline
circuit\_1 &  $2624$ & $35823$ & $3.27e+05$ & circuit simulation problem \\ \hline
ex37 &  $3565$ & $67591$ & $2.26e+02$ & computational fluid dynamics problem \\ \hline
fpga\_trans\_02 &  $1220$ & $7382$ & $2.91e+04$ & subsequent circuit simulation problem \\ \hline
linverse &  $11999$ & $95977$ & $8.24e+03$ & statistical/mathematical problem \\ \hline
msc01440 &  $1440$ & $44998$ & $7.00e+06$ & structural problem \\ \hline
nasa1824 &  $1824$ & $39208$ & $6.63e+06$ & structural problem \\ \hline
nasa2146 &  $2146$ & $72250$ & $3.97e+03$ & structural problem \\ \hline
nasa2910 &  $2910$ & $174296$ & $1.76e+07$ & structural problem \\ \hline
orsirr\_2 &  $886$ & $5970$ & $1.67e+05$ & computational fluid dynamics problem \\ \hline
pde900 &  $900$ & $4380$ & $2.93e+02$ & 2D/3D problem \\ \hline
pores\_3 &  $532$ & $3474$ & $6.64e+05$ & computational fluid dynamics problem \\ \hline
rdb5000 &  $5000$ & $29600$ & $4.30e+03$ & computational fluid dynamics problem \\ \hline
steam1 &  $240$ & $2248$ & $2.99e+07$ & computational fluid dynamics problem \\ \hline
wang1 &  $2903$ & $19093$ & $1.69e+04$ & semiconductor device problem sequence \\ \hline
wang3 &  $26064$ & $177168$ & $1.07e+04$ & semiconductor device problem \\ \hline
watt\_2 &  $1856$ & $11550$ & $1.37e+12$ & computational fluid dynamics problem \\ \hline
young1c &  $841$ & $4089$ & $9.85e+02$ & acoustics problem \\ \hline
    \end{tabular}
    \caption{Properties of the set of test matrices, containing new problems as well as problems related to those in the training set.}
    \label{tab:test_matrices}
\end{table}

\begin{table}
\addtolength{\tabcolsep}{-0.1em}
\begin{tabular}{c|cccccc}
name & PD-GMRES & GMRES(10) & GMRES(20) & GMRES(30) & GMRES(50) & GMRES(100)\\
\hline
ACTIVSg2000     & \textbf{0.8631} & 37.719 & -     & 15.253 & 18.005 & 11.548 \\
CAG\_mat1916    & 0.5497 & -     & -     & -     & -     & \textbf{0.5422}  \\
Chem97ZtZ       & \textbf{0.0050} & 0.0063  & 0.0057  & 0.0062  & 0.0074  & 0.0102  \\
add32           & 0.0185 & \textbf{0.0182}  & 0.0217  & 0.0279  & 0.0357  & 0.0581  \\
adder\_dcop\_01 & 0.2419 & -     & -     & -     & 0.5972  & \textbf{0.2308}  \\
bcsstk34        & \textbf{0.0321} & 0.2869  & 0.1354  & 0.1014  & 0.0751  & 0.0501  \\
bfwa398         & 0.0162 & 0.9326  & 0.1091  & 0.0480  & 0.0586  & \textbf{0.0110}  \\
cavity01        & 0.0100 & 0.0435  & 0.0500  & 0.0338  & 0.0345  & \textbf{0.0067}  \\
cdde2           & 0.0024 & \textbf{0.0020}  & 0.0036  & 0.0053  & 0.0056  & 0.0029  \\
circuit\_1      & \textbf{0.0171} & 0.0397  & 0.0218  & 0.0206  & 0.0201  & 0.0215  \\
ex37            & \textbf{0.0110} & 0.0169  & 0.0126  & 0.0132  & 0.0143  & 0.0192  \\
fpga\_trans\_02 & \textbf{0.0442} & 0.5130  & 0.2972  & 0.1110  & 0.0909  & 0.0628  \\
linverse        & \textbf{7.8090} & -     & -     & -     & 170.61     & 138.54     \\
msc01440        & \textbf{0.8951} & -     & 19.512 & 15.180 & 10.938 & 8.3156  \\
nasa1824        & \textbf{0.7689} & -     & -     & -     & 7.0695  & 3.9878  \\
nasa2146        & \textbf{0.0229} & 0.0676  & 0.0402  & 0.0339  & 0.0328  & 0.0403  \\
nasa2910        & \textbf{2.8038} & -     & 102.56     & -     & 40.150 & 21.559 \\
orsirr\_2       & \textbf{0.0270} & 0.2592  & 0.0785  & 0.0361  & 0.0331  & 0.0329  \\
pde900          & \textbf{0.0028} & 0.0031  & 0.0033  & 0.0041  & 0.0059  & 0.0055  \\
pores\_3        & \textbf{0.2031} & -     & -     & -     & -     & 0.2203  \\
rdb5000         & \textbf{0.2241} & 0.4378  & 0.3218  & 0.2345  & 0.2622  & 0.2556  \\
steam1          & 0.2200 & -     & -     & -     & 53.924     & \textbf{0.1433}  \\
wang1           & 0.1033 & -     & -     & 0.0729  & \textbf{0.0568}  & 0.0844  \\
wang3           & \textbf{0.4873} & 0.6238  & 0.6429  & 0.6786  & 0.9127  & 1.1876  \\
watt\_2         & \textbf{0.0478} & 0.5080  & 0.3434  & 0.2750  & 0.3360  & 0.3002  \\
young1c         & \textbf{0.1954} & 0.7906  & 0.6262  & 0.7526  & 0.6941  & 0.5830 
\end{tabular}
\caption{Runtimes for optimized PD-GMRES compared to various GMRES($m$) for the test matrices. For each problem the best performing algorithm is marked in bold. Missing entries indicate failure to converge within 240 seconds.}
\label{tab:test_gmres_comp}
\end{table}

Table~\ref{tab:test_gmres_comp} shows the same kind of comparison as Table~\ref{tab:training_gmres_comp}, now for the test matrices. One finds that there are more matrices for which some or most GMRES($m$) fail to converge in the given time span of $240$ seconds. As found for the training matrices, some of these convergence failures are indicative of stagnation, meaning that no amount of permitted runtime would yield convergence. And in the cases where the iteration does not stagnate, the rate of convergence tends to be very low. There is also an additional number of problematic matrices for which GMRES($100$) outperforms the optimized PD-GMRES.

\begin{table}
\centering
\begin{tabular}{|cccccc|}
	\hline
	PD-GMRES & GMRES(10) & GMRES(20) & GMRES(30) & GMRES(50) & GMRES(100)\\
	\hline
	$1$ & $4.5065$ & $3.0607$ & $2.3228$ & $3.6387$ & $2.0761$\\
	\hline
\end{tabular}
\caption{Geometric means of normalized data for the test matrices (excluding non-converging matrices column-wise).}
\label{tab:test_gmres_geomeans}
\end{table}

The geometric means of the data, normalized around PD-GMRES performance, are shown in Table~\ref{tab:test_gmres_geomeans}. As previously in Table~\ref{tab:training_gmres_geomeans}, the means are calculated column-wise with non-converging entries being discarded, which means that these values can not be used as a direct comparison between different GMRES($m$). One can observe that PD-GMRES performs significantly better than even the best GMRES($m$) (which is again GMRES($100$)). There are two matrices for which GMRES($10$) is actually the best performing algorithm. However, in both of these cases the advantage over the optimized PD-GMRES is very small so they are not treated separately. Additionally, due to the number of missing entries (for which GMRES($m$) did not converge), the calculated geometric means actually understate how poor GMRES performance can be expected to be on average.

Overall we find that the optimized PD-GMRES performs very well for the whole set of test matrices, despite having been optimized for a different matrix set. Because of this we can recommend using PD-GMRES with $m_{\text{init}} = 10, m_{\min} = 3, m_{\text{step}} = 10, \alpha_p = -0.625, \alpha_d = 4.375$ as an iterative solver for a wide set of matrices. Given that we found problematic matrices where this PD-GMRES performed worse than GMRES($100$) we apply our optimization procedure to those matrices separately.

\subsection{Classification and optimization of problematic matrices}\label{sec:problematic}
We found matrices in both the training and the test set for which GMRES($100$) outperformed the optimized PD-GMRES. These problematic matrices are evaluated separately, both to further test the effectiveness of the optimization procedure and to provide an alternative set of PD-GMRES parameters to be used for similar matrices. The problematic training matrices are \textit{cdde1}, \textit{circuit\_2}, \textit{rdb2048}, \textit{sherman4}, and \textit{steam2}. The problematic test matrices are \textit{adder\_dcop\_01}, \textit{bfwa398}, \textit{CAG\_mat1916}, \textit{cavity01}, \textit{steam1}, and \textit{wang1}.

Figure~\ref{fig:clustering} shows the location of these matrices in a scatter plot showing GMRES($100$) performance against the matrix condition number. The apparent clustering of the problematic matrices suggests some similarities between them. It would be desirable to find a way of classifying these clusters without having to calculate GMRES($100$) to convergence for a given matrix.

\begin{wrapfigure}{r}{0.56\textwidth}
\centering
\includegraphics[width=0.535\textwidth]{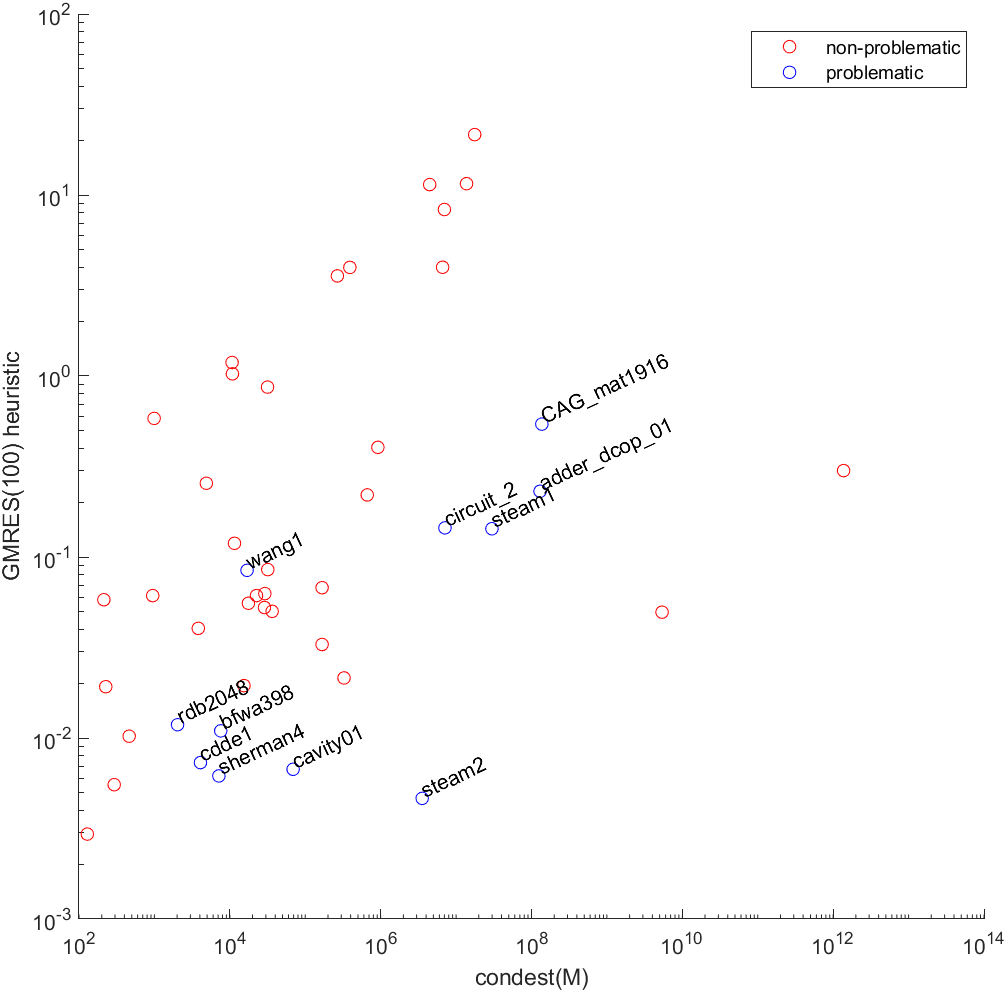}
\caption{GMRES(100) performance (logarithmically) over matrix condition. For the named problems with blue marker GMRES(100) outperformed the optimized PD-GMRES.}
\label{fig:clustering}
\end{wrapfigure}

A natural consideration would be to perform only a single iteration of GMRES($100$), hoping that the initial change in the residual is indicative of the performance of the full iteration. For the training matrices, classification is easily possible by assessing the residual after a single iteration of GMRES($100$) and $10$ iterations of GMRES($10$). Factoring these values with the runtime for these iterations also gives a valid classifier.

However, applying this classifier to the test matrices as well as all other attempts of creating a simple classifier by evaluating residuals and runtimes (for a single iteration of GMRES($100$) or $10$ iterations of GMRES($10$)) as well as the matrix norm and condition number were unsuccessful for the test matrices. Specifically, the test matrix set contains several matrices for which GMRES(100) and GMRES(10) actually perform similarly well initially (when factoring in runtime), which is not the case for the training matrices. This does not mean that there is no efficient classifier for the problematic matrices, however, it does suggest that finding one might be difficult. Given the relatively small amount of problematic matrices any classifier using more complex evaluations than those considered here might be prone to overfitting and thereby fail to classify additional similarly problematic matrices.

While it is unfortunate that no efficient classifier was found for separating problematic from non-problematic matrices, we nevertheless perform a new optimization procedure to find better performance for the problematic matrices. For this, we perform the optimization procedure for the problematic training matrices. It is not to be expected that this will actually provide best performance for the problematic test matrices, since the ones for training do seem to behave differently as was found when trying to classify them.

\begin{figure}
\centering
\includegraphics[width=\textwidth]{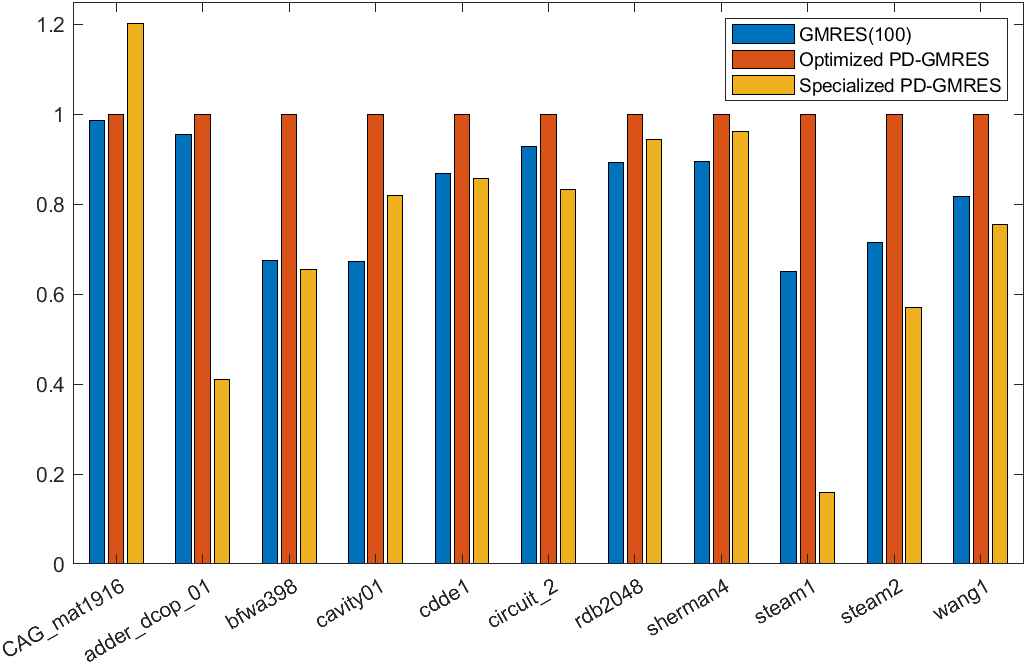}
\caption{Comparing performance of GMRES(100) and the previously optimized PD-GMRES against the specialized PD-GMRES found by optimizing the problematic training matrices. Normalizing around the optimized PD-GMRES the geometric means are: $0.8147$ for GMRES(100), $1$ for optimized PD-GMRES, $0.6676$ for the specialized PD-GMRES.}
\label{fig:problematic_optimized}
\end{figure}

The result of the optimization sequence, evaluated for $m_{\text{init}} = 30$ and $m_{\text{init}} = 50$ (smaller $m_{\text{init}}$ are unlikely to be useful given that the problematic matrices perform well for GMRES($100$)), gives a specialized PD-GMRES with optimal parameters found to be $m_{\text{init}} = 30, m_{\min} = 33, m_{\text{step}} = 39, \alpha_p = -42.5, \alpha_d = 0$. This means that the derivative term \eqref{eq:derivative_term} in the controller is effectively discarded and $m_1$ is effectively $69$ as the reset condition is triggered in the first PD-GMRES iteration. Evaluating the quadtree sequence reveals that the second optimization in the $(\alpha_p, \alpha_d)$-space does not provide any meaningful distinction between the values. The algorithm then chooses a corner point in the quadtree (to better align with actually calculated values in low resolutions) which yields $\alpha_d = 0$.

The geometric means for the optimization sequence in Figure~\ref{fig:problematic_optimized} show that on average, the specialized PD-GMRES performs significantly better than the previous optimized PD-GMRES and outperforms GMRES($100$) as well. Improvements over the previous PD-GMRES are found for ten out of the eleven problematic matrices, and while there are three matrices for which GMRES($100$) is still better than the new PD-GMRES, the differences there are small, especially when compared to the gains found for some of the other matrices.

This is a remarkably positive result, considering that the problematic training matrices are clearly not a good model for the general kind of problematic matrices found in the test set, and further considering that the algorithm chooses a somewhat arbitrary pair of values in the $(\alpha_p, \alpha_d)$-space. Being a well-performing algorithm despite these shortcomings is a positive indicator for the reliability of the optimization procedure.

Performing an additional optimization for just the problematic test matrices can provide even better performance. Compared to the optimized PD-GMRES we find geometric means of $0.4330$ and $0.4021$ for $m_{\text{init}} = 30$ and $50$ respectively. This means that optimizing the set of problematic test matrices yields a PD-GMRES which outperforms the specialized one by a significant margin. This is not surprising, as this optimization procedure does in fact differentiate the $(\alpha_p, \alpha_d)$-space and selects a specific, optimal value pair. Optimal parameters are found to be $m_{\text{init}} = 50, m_{\min} = 51, m_{\text{step}} = 76, \alpha_p = -41.25, \alpha_d = 29.375$. Unlike previously, $\alpha_d$ is not zero. We do find similar behavior where $m_{\text{init}}$ is smaller than $m_{\min}$ which results in a jump in the first iteration.

With the problematic training matrices optimized, it is natural to perform another optimization procedure for the non-problematic training matrices and evaluate the resulting parameters for only the non-problematic training and test matrices. Doing so gives the same parameters for $m_{\text{init}} = 10$ as the initial optimized PD-GMRES. For higher $m_{\text{init}}$ the other parameters differ from the initial evaluation, however, PD-GMRES performs worse for those than it did for $m_{\text{init}} = 10$. This suggests that the initial optimization already gave good performance for the non-problematic matrices anyways, rendering additional optimization less relevant.

\revision{\subsection{Large matrix test}\label{sec:large}
Practical problems often involve larger matrices than those evaluated so far and also require the use of preconditioners. We select a new set of matrices from \cite{MatrixDatabase}, focusing on large matrices (roughly an order of magnitude larger than previous ones), which all relate to computational fluid dynamics problems. For all matrices in this set (cf.~Table~\ref{tab:large_matrices}) GMRES($m$) does not converge within a reasonable time when not using a preconditioner, so preconditioning is necessary. We choose ILU-based preconditioners for all large matrices, using ILU(0) (zero-fill) for problems without zeros on the diagonal and ILUTP (thresholding and pivoting) \cite{Saad2003} for the remaining ones. Here, we reduce the threshold parameter $\tau$ with increasing matrix size to keep the runtime of all matrices somewhat similar. By comparing our optimized PD-GMRES (cf.~Section~\ref{sec:training data}) to GMRES($m$) we find that, for a majority of the large matrices, GMRES($100$) performs better than PD-GMRES. This means that by our classification these matrices are considered problematic therefore suitable for applying our specialized PD-GMRES (cf.~Section~\ref{sec:problematic}). Table~\ref{tab:large_gmres_comp} shows that the specialized PD-GMRES outperforms all considered GMRES($m$), even for non-problematic large matrices where the optimized PD-GMRES already performs well. The geometric mean of GMRES($100$) runtime ($m=100$ being the only value for which GMRES($m$) converges for all large matrices) relative to the runtime of the specialized PD-GMRES is $1.6006$.

\begin{table}[h]
\centering
\revision{
\begin{tabular}{|c|c|c|c|c|}
\hline
Name & n & Non-zeros & Condition & Preconditioner \\ \hline
Goodwin\_040 &  $17922$ & $561677$ & $2.56e+06$ & ILUTP, $\tau = 0.05$ \\ \hline
Goodwin\_054 &  $32510$ & $1030878$ & $6.30e+06$ & ILUTP, $\tau = 0.04$ \\ \hline
Goodwin\_071 &  $56021$ & $1797934$ & $1.43e+07$ & ILUTP, $\tau = 0.03$ \\ \hline
Goodwin\_095 &  $100037$ & $3226066$ & $3.43e+07$ & ILUTP, $\tau = 0.02$ \\ \hline
Goodwin\_127 &  $178437$ & $5778545$ & $8.20e+07$ & ILUTP, $\tau = 0.01$ \\ \hline
venkat25 &  $62424$ & $1717763$ & $1.22e+08$ & ILU(0) \\ \hline
venkat50 &  $62424$ & $1717777$ & $5.32e+08$ & ILU(0) \\ \hline
raefsky3 &  $21200$ & $1488768$ & $5.36e+11$ & ILU(0) \\ \hline
water\_tank &  $60740$ & $2035281$ & $2.52e+09$ & ILU(0) \\ \hline
poisson3Db &  $85623$ & $2374949$ & $1.66e+05$ & ILU(0) \\ \hline
\end{tabular}}
\caption{\revision{Properties of the set of large matrices which are evaluated with preconditioners. All matrices are related to computational fluid dynamics problems. Since the \textit{Goodwin} matrices have zeros on the diagonal, ILU(0) is not applicable and we choose ILUTP instead. We lower the ILUTP threshold $\tau$ with increasing matrix size to keep the PD-GMRES runtime balanced.}}
\label{tab:large_matrices}
\end{table}

\begin{table}
\revision{
\begin{tabular}{c||cc|ccccc}
 & \multicolumn{2}{c|}{PD-GMRES} & \multicolumn{5}{c}{GMRES($m$)} \\
Name & Optimized & Specialized & $m = 10$ & $m = 20$ & $m = 30$ & $m = 50$ & $m = 100$\\
\hline
Goodwin\_040 & 1.9403 & \textbf{0.3513} & - & - & - & 1.6612 & 0.6593\\
Goodwin\_054 & 5.5042 & \textbf{0.7044} & - & - & - & - & 0.9228\\
Goodwin\_071 & 1.3334 & \textbf{0.4419} & - & - & 3.5029 & 1.1785 & 0.7436\\
Goodwin\_095 & 8.2934 & \textbf{1.4191} & - & - & 10.294 & 5.1581 & 1.9296\\
Goodwin\_127 & 46.832 & \textbf{3.8337} & - & - & - & - & 8.0535\\
venkat25 & \textbf{0.5007} & 0.6679 & 0.7832 & 0.8158 & 0.8721 & 0.9988 & 1.1896\\
venkat50 & \textbf{0.7156} & 0.9495 & 1.1685 & 1.1585 & 1.2641 & 1.4927 & 1.8642\\
raefsky3 & 0.2240 & \textbf{0.1861} & 1.5778 & 0.4343 & 0.3713 & 0.3332 & 0.2732\\
water\_tank & 19.893 & \textbf{2.9135} & - & - & - & - & 4.1406\\
poisson3Db & \textbf{0.6422} & 0.6984 & 1.4529 & 1.1387 & 0.9429 & 0.9364 & 0.8939
\end{tabular}}
\caption{\revision{Runtimes for the set of large matrices using GMRES($m$) with $m = 10, 20, 30, 50, 100$ and PD-GMRES with optimized and specialized parameters from previous sections. For each problem the best performing algorithm is marked in bold. Missing entries indicate failure to converge within $240$ seconds.}}
\label{tab:large_gmres_comp}
\end{table}

For a comparison of the total cost of solving the linear system, the cost of calculating the ILU-based preconditioners would have to be added. We did not include this cost in our comparison as the preconditioners were calculated once in advance and were also not specifically optimized beyond general suitability.}

\subsection{Evaluating a different parameter quadrant} \label{subsec:2q}
All previous analysis focused on the $(\alpha_p < 0, \alpha_d > 0)$-quadrant. We now turn to the $(\alpha_p > 0, \alpha_d > 0)$-quadrant which is evaluated for $m_{\text{init}} = 10$ and $m_{\text{init}} = 30$ using \revision{the same training and test matrix sets,} the same initial choices for $m_{\min}$ and $m_{\text{step}}$ and the same quadtree size and resolution as before, but with $\alpha_p$ being evaluated in the interval $[0, 40]$.

\begin{wrapfigure}{r}{0.6\textwidth}
\begin{subfigure}{0.3\textwidth}
\centering
\includegraphics[width=\textwidth]{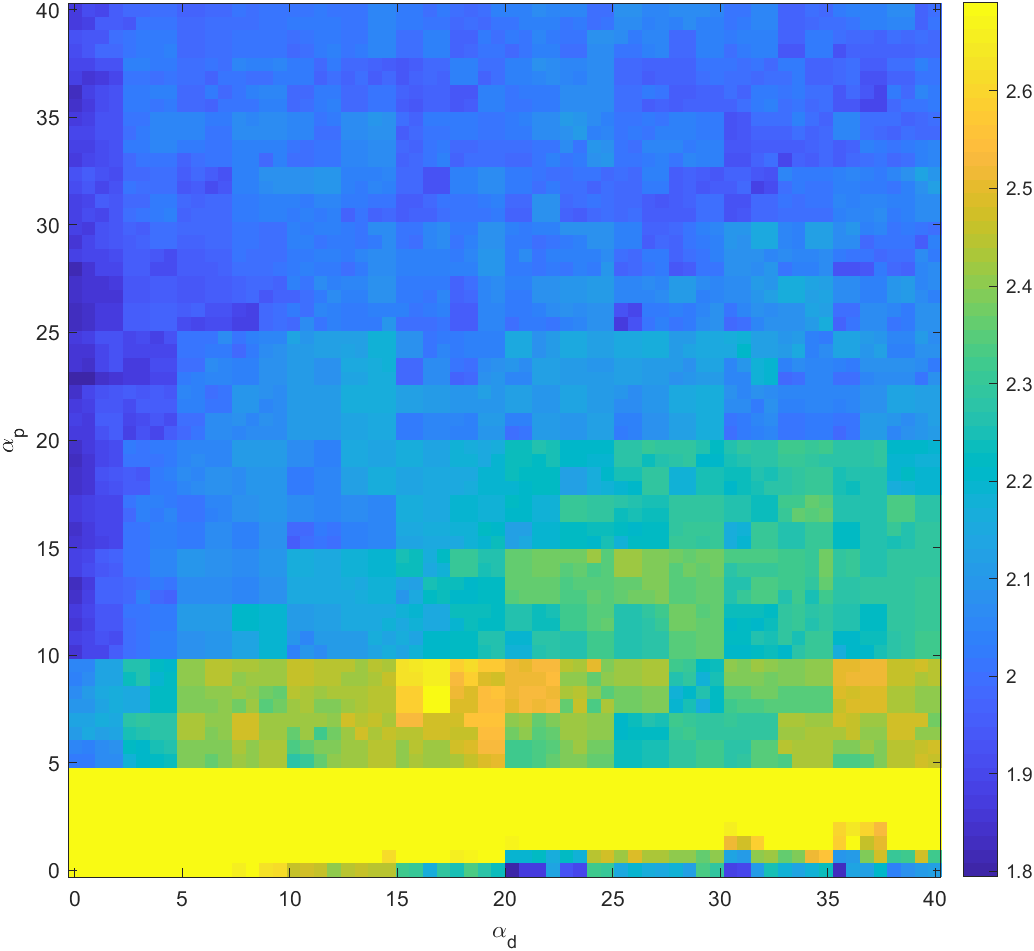}
\end{subfigure}
\begin{subfigure}{0.3\textwidth}
\centering
\includegraphics[width=\textwidth]{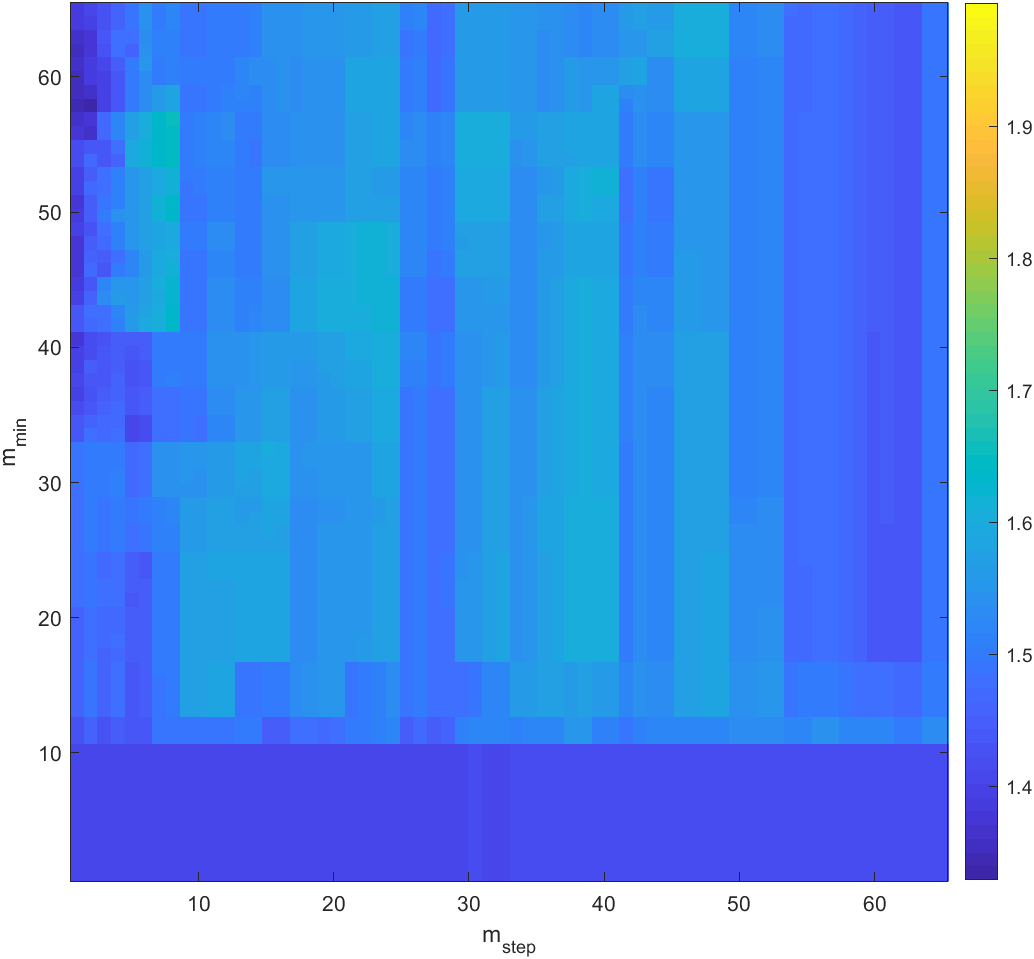}
\end{subfigure}
\caption{Single cycle of the optimization sequence for $m_{\text{init}} = 10$ in the $(\alpha_p > 0, \alpha_d > 0)$-quadrant.}
\label{fig:2q_10_qts}
\end{wrapfigure}

Figure~\ref{fig:2q_10_qts} shows a range of parameters with $\alpha_p$ close to zero, for which PD-GMRES performance is poor. This can be related to Proposition~\ref{prop:2} which would guarantee a good rate of convergence if the restart parameter $m$ does not change and $\frac{2\alpha_p}{\alpha_d}$ is large. For very large $\alpha_d$ one can find good convergence again. As those values would correspond to poor convergence in the case of $m$ not changing, the fact that one finds good convergence suggests that here $m$ is actually pushed to decrease by the controller, giving convergence through the reset condition similarly to the behavior in the original $(\alpha_p < 0, \alpha_d > 0)$-quadrant.

Optimal parameters for $m_{\text{init}} = 30$ were found to be $m_{\min} = 10, m_{\text{step}} = 5, \alpha_p = 0, \alpha_d = 21.875$. Since $m_{\min}$ and $m_{\text{step}}$ did not change from their initial choices during their optimization step, we stop the optimization sequence after one cycle. Optimal parameters for $m_{\text{init}} = 10$ are $m_{\min} = 57, m_{\text{step}} = 2, \alpha_p = 22.5, \alpha_d = 0.625$, after one optimization cycle. Due to poor performance (both in terms of the final parameters not performing well and in terms of the actual optimization requiring more time than it did in the original quadrant) no further optimization cycles are calculated.

As an example of the kind of behavior one can find in this quadrant, we examine \textit{circuit\_2} in Figure~\ref{fig:2q_circuit_2}. The top-left subfigure shows the quadtree, revealing good performance for both $\alpha_p \gg 0, \alpha_d \gtrsim 0$ and for $\alpha_d \gg 0, \alpha_p \gtrsim 0$. This is unusual, as most matrices show good performance for just one of these parameter regions.
\begin{wrapfigure}{r}{0.6\textwidth}
\begin{subfigure}{0.3\textwidth}
\centering
\includegraphics[width=\textwidth]{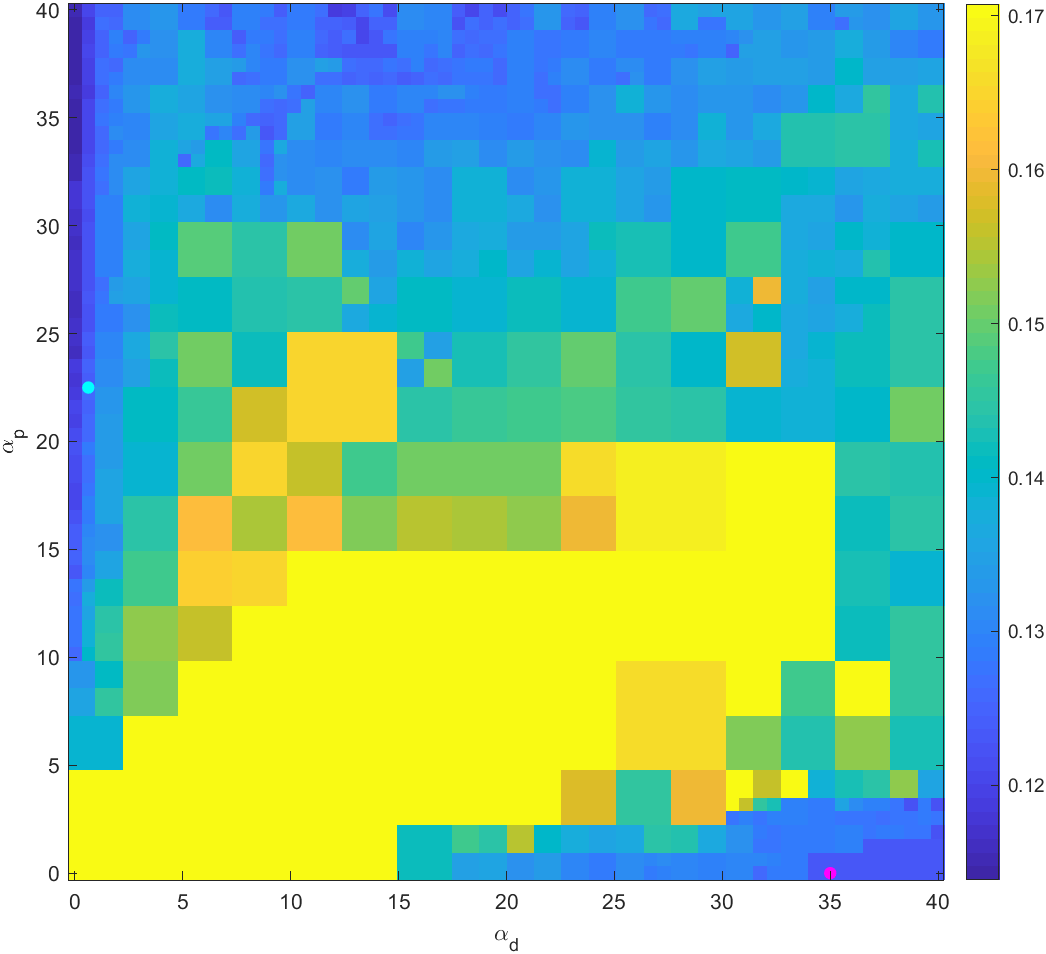}
\end{subfigure}
\begin{subfigure}{0.3\textwidth}
\centering
\includegraphics[width=\textwidth]{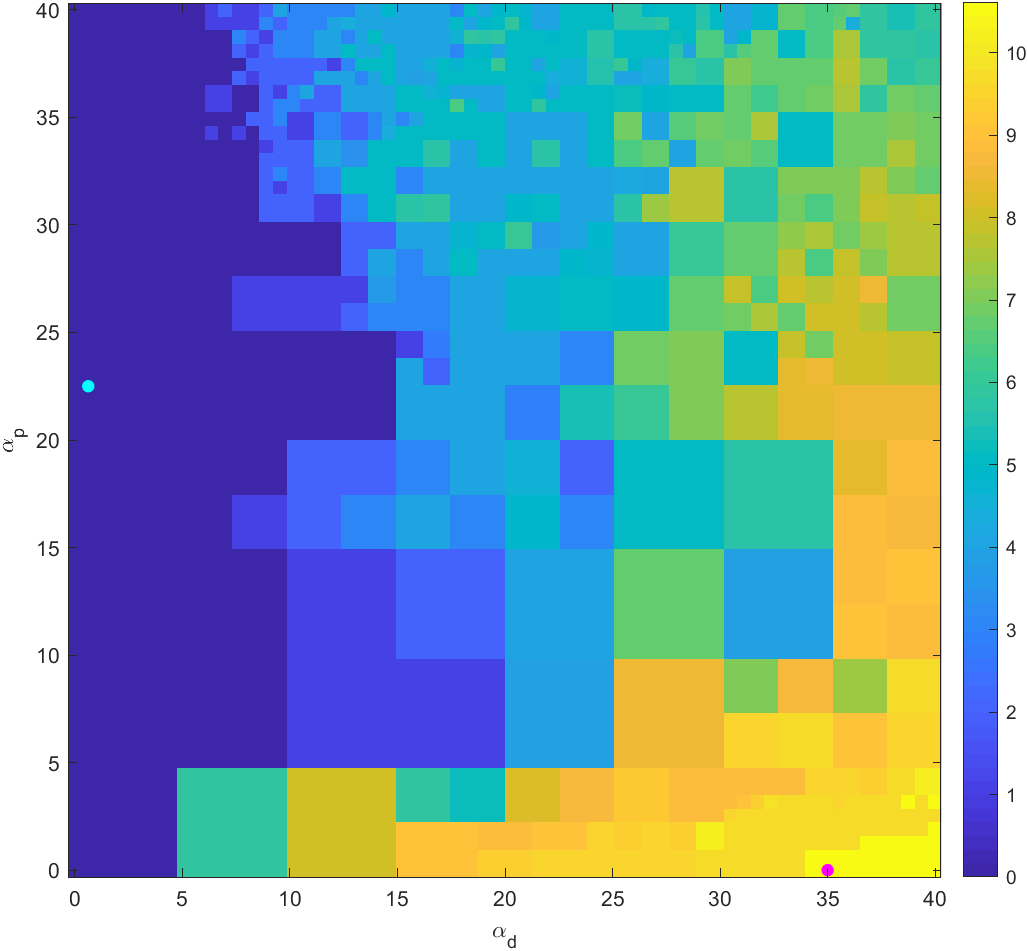}
\end{subfigure}
\begin{subfigure}{0.3\textwidth}
\centering
\includegraphics[width=\textwidth]{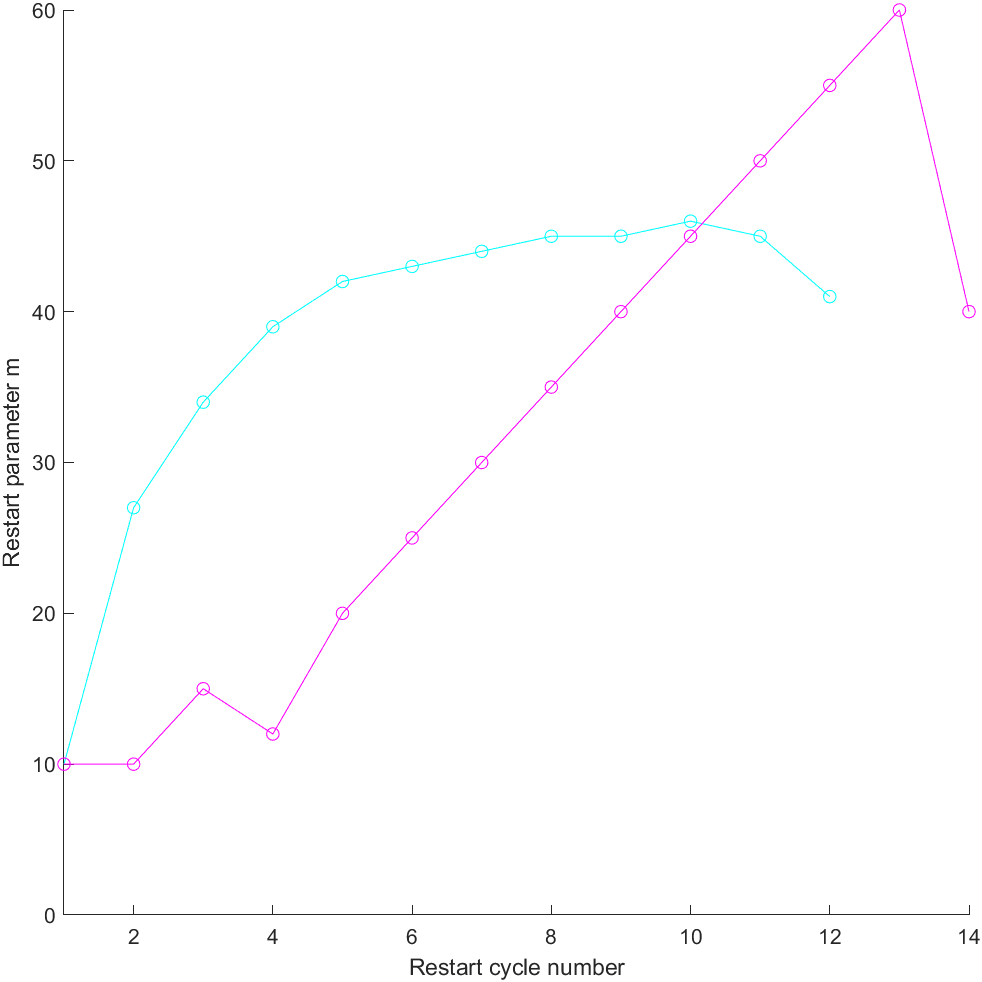}
\end{subfigure}
\begin{subfigure}{0.3\textwidth}
\centering
\includegraphics[width=\textwidth]{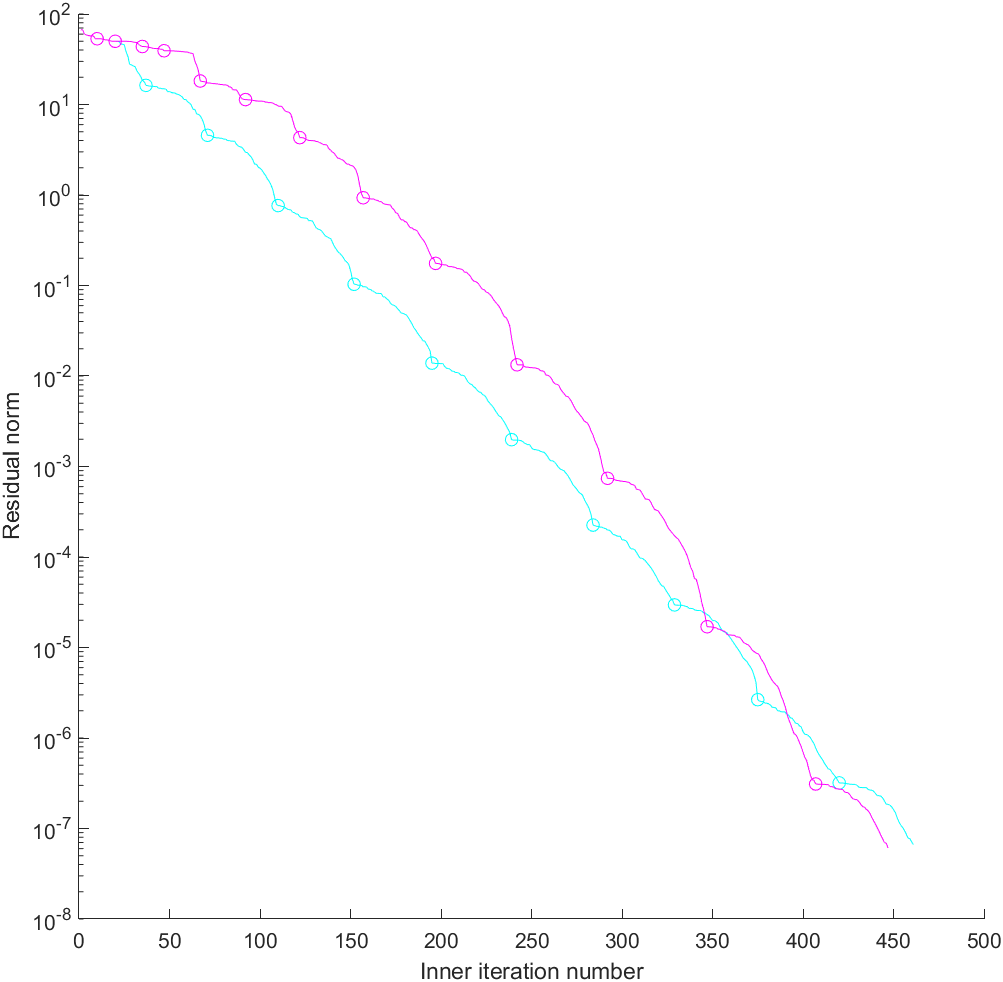}
\end{subfigure}
\caption{Quadtree, minima, restart parameters and residual norms for \textit{circuit\_2} in the $(\alpha_p > 0, \alpha_d > 0)$-quadrant.}
\label{fig:2q_circuit_2}
\end{wrapfigure}
The top-right subfigure shows the number of times the reset condition in the PD-GMRES algorithm is triggered by the calculated restart parameter $m$ being less than $m_{\min}$. 
As expected, for $\alpha_p \gg 0, \alpha_d \gtrsim 0$ the restart parameter grows and is never less than $m_{\min}$, while for $\alpha_d \gg 0, \alpha_p \gtrsim 0$ we see $m < m_{\min}$ at almost every restart. This behavior is illustrated in the bottom-left subfigure, which shows the evolution of the restart parameter $m$. The two different types of behavior can be clearly seen, and the cyan iteration shows nearly constant values for $m$ which comes close to the condition described in Proposition~\ref{prop:2}. Finally, the bottom-right subfigure shows the residual norm during the iteration. The two iterations have nearly identical runtime ($0.1444$ for cyan versus $0.1406$ for magenta), which shows that in principle one can expect either parameter region to offer good performance.

When analyzing individual matrices, we find rather localized minima for many of them (rather than large areas of well-performing parameters). However, comparing the individual optimal values for $\alpha_p$ and $\alpha_d$ for each matrix against the performance of the optimized PD-GMRES gives a geometric mean of $0.9005$ for $m_{\text{init}} = 10$ and $0.9113$ for $m_{\text{init}} = 30$. These values are worse than the one found in the $(\alpha_p < 0, \alpha_d > 0)$-quadrant, which is $0.8514$. Excluding two outliers (\textit{cavity10} and \textit{sherman5}), for which the $(\alpha_p > 0, \alpha_d > 0)$-quadrant offers no well-performing values does improve these geometric means to about $0.83$, but this still gives no indication that local minima found in this quadrant actually perform significantly better than those found in the original $(\alpha_p < 0, \alpha_d > 0)$-quadrant, which means that the good rate of convergence theoretically found in the case of Proposition~\ref{prop:2} appears uncommon to occur in practice.

\begin{wraptable}{l}{0.4\textwidth}
\centering
\begin{tabular}{|c|c|c|}
	\hline
	& \multicolumn{2}{c|}{Matrix set}\\
%	\cline{2-3}
	\hline
	$m_{\text{init}}$& training & test \\
	\hline
	$10$ & $1.6654$ & $1.2793$\\
	\hline
	$30$ & $1.2645$ & $1.4077$\\
	\hline
\end{tabular}
\caption{Geometric means for PD-GMRES optimized in the $(\alpha_p > 0, \alpha_d > 0)$-quadrant, relative to the optimized PD-GMRES with $\alpha_p < 0$.}
\label{tab:2q_table}
\end{wraptable}

Table~\ref{tab:2q_table} shows that the best-performing PD-GMRES found in the  $(\alpha_p > 0, \alpha_d > 0)$-quadrant for $m_{\text{init}} = 10$ actually performs better for the test matrices compared to the training matrices. However, it is still inferior and the individual matrix data suggests that this is a general trend, i.e., while there are some extreme outliers, there are only a few matrices where the PD-GMRES with $\alpha_p > 0$ shows a significant performance increase, and a much greater number of matrices where it performs slightly or notably worse.

It must also be noted that convergence is not always achieved during the quadtree evaluation, specifically in the large yellow area in Figure~\ref{fig:2q_10_qts} for small $\alpha_p$ we often failed to find convergence in a reasonable time span. As parameters in this area end up not being evaluated further this is not a problem for the algorithm, except for increasing its runtime.

\subsection{An algorithmic extension for PD-GMRES} \label{subsec:extend}
When introducing the training matrices, cf.~Table~\ref{tab:training_matrices}, we discarded three matrices for which both GMRES and PD-GMRES were very slow to converge and thus would have been time-intensive to include in the optimization procedure. Two of these matrices, \textit{sherman3} and \textit{wang4}, do not even come close to convergence in practical time spans (the residual norm is found to still be in the order of $10^0$ after $60$ seconds of GMRES or PD-GMRES iteration). This is likely due to the very large condition number of \textit{sherman3}, which could be improved by using preconditioning, and the rather large size and number of non-zeros for \textit{wang4} which increase the runtime for GMRES. When using PD-GMRES for these matrices the algorithm ends up using very large restart parameters $m$, due to the reset condition in Algorithm~\ref{alg:pd_controller} being unrestrained in terms of how large $m$ can get. While this behavior is necessary to have a theoretical guarantee of convergence (cf.~Proposition~\ref{prop:1}), it may result in poor performance.

For the third discarded matrix, \textit{ex40}, we did find convergence when using PD-GMRES. To improve the performance for this matrix we introduce an extension of the algorithm, using an additional parameter $m_{\max}$.

\begin{algorithm}
\caption{Extended reset condition}
\label{alg:pdgmres_extended}
\begin{algorithmic}
\STATE \hspace*{-\algorithmicindent}\textbf{Parameters:} $m_{\text{init}}, m_{\min}, m_{\text{step}}, m_{\max}$
\REQUIRE{$m_{j+1}, c$}
\ENSURE{$m_{j+1}, c$}
\IF{$m_{j+1} < m_{\min}$}
\STATE{$c = c+1$}
\STATE{$m_{j+1} = m_{\text{init}}+c\cdot m_{\text{step}}$}
\ENDIF
\IF{$m_{j+1} > m_{\max}$}
\STATE{$c = 0$ \hfill\COMMENT{Re-initialize reset counter}}
\STATE{$m_{j+1} = m_{\text{init}}$}
\ENDIF
\end{algorithmic}
\end{algorithm}

The modified reset condition (Algorithm~\ref{alg:pdgmres_extended}) constrains the restart parameter $m$ into the interval $m_{\min} \leq m \leq m_{\max}$. While this means that it is possible for the iteration to stagnate, as Proposition~\ref{prop:1} no longer applies, it can also provide better performance due to lower restart parameters being used. We test this by applying the extended algorithm to \textit{ex40}.

\begin{figure}
\begin{subfigure}{0.48\textwidth}
\centering
\includegraphics[width=\textwidth]{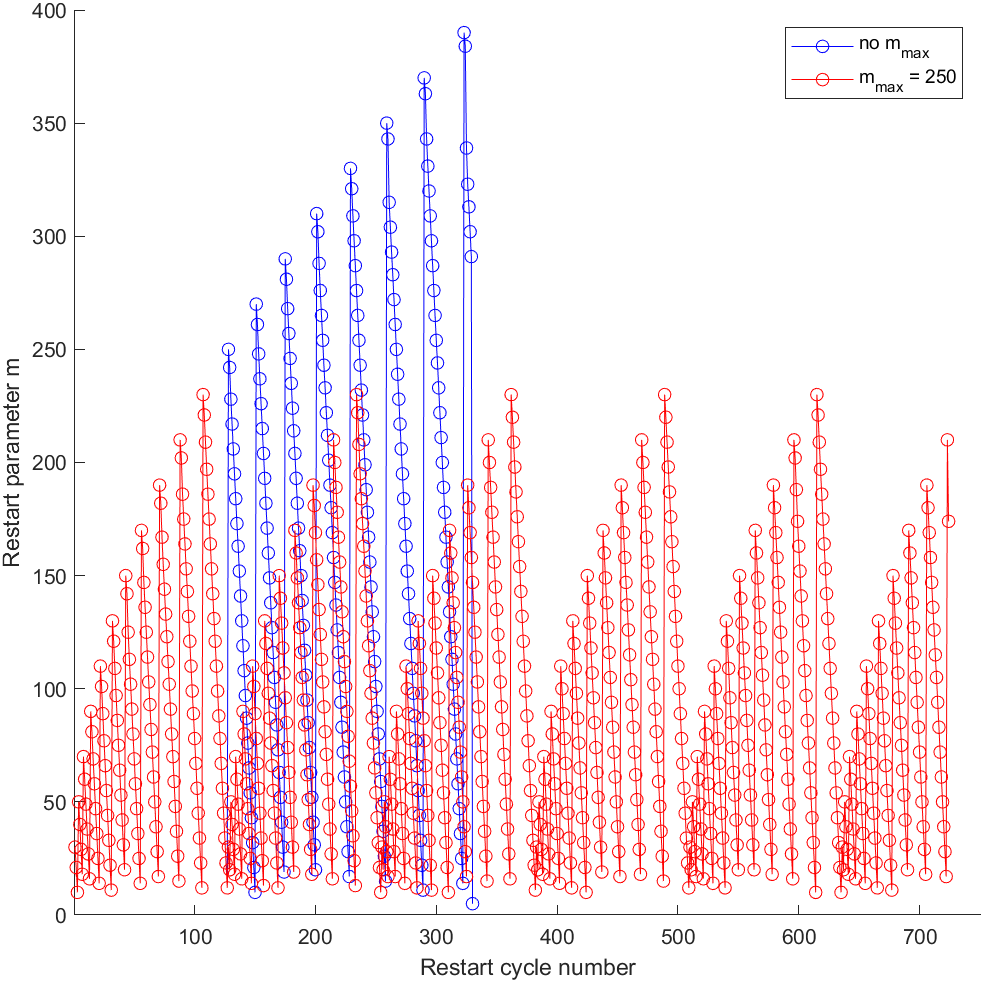}
\end{subfigure}
\begin{subfigure}{0.48\textwidth}
\centering
\includegraphics[width=\textwidth]{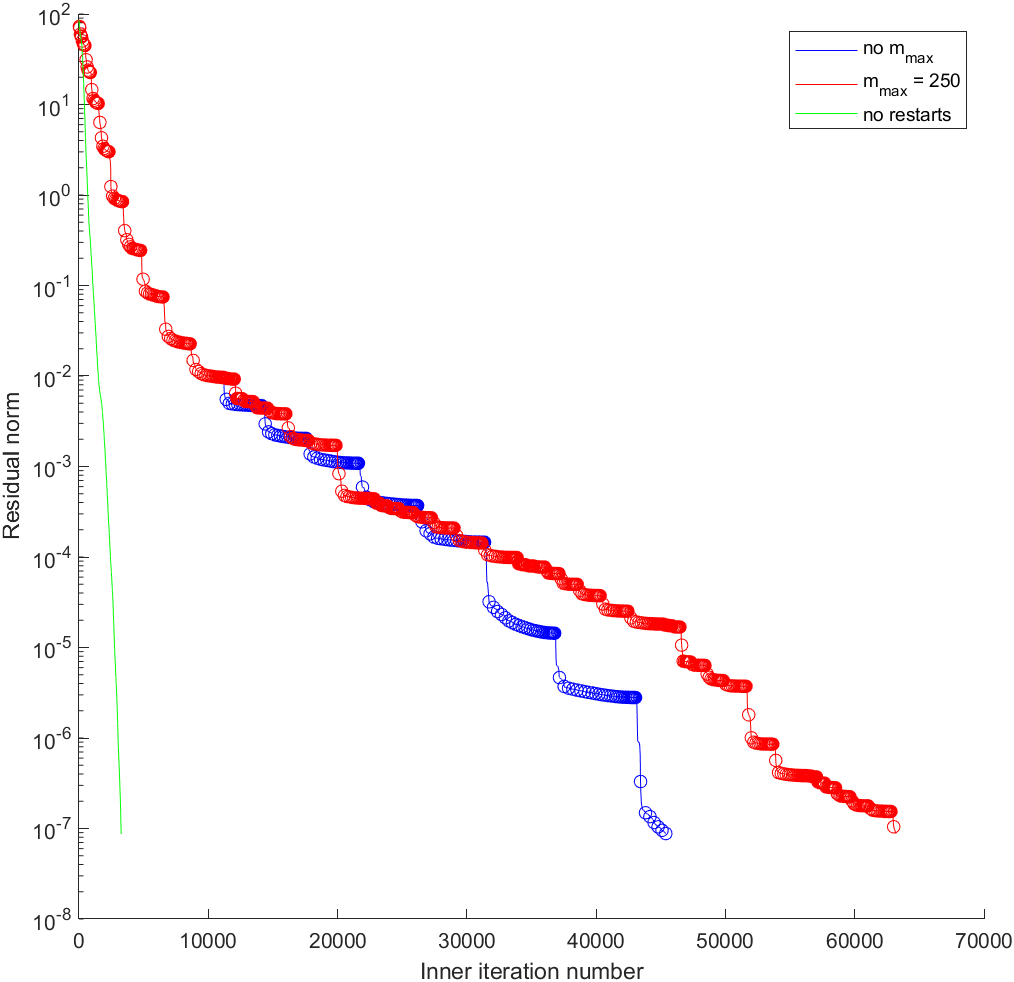}
\end{subfigure}
\caption{Restart parameter $m$ and residual norm for \textit{ex40} with and without $m_{\max} = 250$.}
\label{fig:ex40}
\end{figure}
Using $m_{\text{init}} = 30, m_{\min} = 10, m_{\text{step}} = 20, \alpha_p = -10, \alpha_d = 10$ (unoptimized, arbitrary choice), we compare the behavior of the original PD-GMRES to the extended algorithm with $m_{\max} = 250$ in Figure~\ref{fig:ex40}. We find better performance for the extended algorithm, with a runtime of $78.73$ seconds with $m_{\max} = 250$, $92.41$ seconds without $m_{\max}$ and $97.35$ for using GMRES without restarts.

This means that using $m_{\max}$ can improve performance for matrices for which PD-GMRES converges very slowly. It is important to note that choosing $m_{\max}$ too small can result in significantly worse performance for matrices which require large $m$ for good convergence. There does not seem to be an obvious method for determining how small $m_{\max}$ can be to optimize performance. It should also be noted that applying the extended algorithm to most of the training matrices does not result in any significant performance gain, as PD-GMRES usually converges quickly and thereby we have a short iteration length with the values for $m$ staying small enough to have good performance anyways.

\subsection*{Implementation details}
\revision{All data presented here was calculated using MATLAB version R2019a on an Intel i7-12700k CPU, except for Section~\ref{sec:large} for which an AMD Ryzen 7 9800X3D CPU was used. This means that the runtime values shown in Table~\ref{tab:large_gmres_comp} cannot be directly compared to those in Table~\ref{tab:training_gmres_comp} or Table~\ref{tab:test_gmres_comp}.}

Using MATLAB version 2022b or newer will give different results, as the \textit{mldivide} operator was optimized for small matrices in that version, resulting in different numerical behavior. Specifically, GMRES($m$) calls for $m \leq 16$ give different results in MATLAB version 2022b and newer. Over a full GMRES or PD-GMRES iteration this may result in noticeable differences in terms of iteration length and runtime.

\section{Conclusions} \label{sec:Conclusion}
We have presented a quadtree-based geometric method for solving two-dimensional optimization problems and have applied this method to the problem of finding optimal parameters for PD-GMRES, such that the runtime of PD-GMRES is minimized. Performing the optimization procedure on a set of training matrices we have found well-performing PD-GMRES parameters which were then validated on a set of test matrices. \revision{We expect that these parameters will perform well for other matrices and could be considered candidates for standard parameters to be used for PD-GMRES.} We have further confirmed the effectiveness of our optimization approach by applying it to a subset of training matrices which exhibited problematic behavior. This application has provided even better PD-GMRES performance and has shown that the quadtree method is a flexible tool for optimizing PD-GMRES parameters for sets of matrices. \revision{Moreover, a large matrix test has indicated the suitability of the obtained parameters for PD-GMRES in combination with preconditioning.} Employing the optimization procedure to a different parameter quadrant did not result in better performance. Attempts to find a simple classifier for the problematic matrices were unsuccessful. We have further introduced an extension of the PD-GMRES algorithm which uses an additional parameter to constrain the restart parameter in a fixed range. We have applied this modified algorithm to a specific matrix for which PD-GMRES has long runtime and have found improved performance.

% \section{Outlook}
We identify several directions for further investigations. 
% The quadtree method can be applied to different optimization problems. Additionally, it can be compared to other optimization methods in terms of runtime. 
As systems of linear equations are often preconditioned before solving, the effect of preconditioning on the evaluated matrix sets should be \revision{further} investigated. One specific question would be whether there exist any clear trends for how the PD-GMRES parameters should be changed when switching from an un-preconditioned system to the same system with a preconditioner. Further attempts could be made to classify the problematic matrices. As we have found well-performing PD-GMRES parameters for both the non-problematic and the problematic matrices, an efficient classifier would allow for an a-priori determination of which parameters should be used for a given matrix. However, it seems likely that there is no such universal classifier. Alternative criteria for the construction of the quadtree could be considered, as briefly outlined at the end of Section \ref{sec:quadtree}. The extended PD-GMRES algorithm using the additional parameter $m_{\max}$ should be further investigated. If some practical estimates for $m_{\max}$ can be found, the extended algorithm could be generally superior to the original. Here, matrices for which PD-GMRES performs very poorly should be of particular focus. It is also possible to apply the optimization procedure to PD-GMRES variants which make use of augmentation or deflation of the Krylov subspace to improve the rate of convergence \cite{Cabral2020}.

\section*{Acknowledgement}
Funded by Deutsche Forschungsgemeinschaft (DFG, German Research Foundation) under Germany’s Excellence Strategy (Project Number 390740016 -- EXC 2075) and the Collaborative Research Center 1313 (Project Number 327154368 -- SFB 1313). We acknowledge the support by the Stuttgart Center for Simulation Science (SC SimTech).

\bibliographystyle{abbrvnat}
{\bibliography{bibliography.bib}}

\section*{\revision{Appenix A: Total runtime for optimization}}
\revision{In this appendix we present an example of the total cost of our optimization procedure. For the three optimization cycles for the training matrices using $m_{init} = 10$ (cf. Figure~\ref{fig:optimization_sequence}) we have an approximate total runtime of $7330.3$ seconds. In Table~\ref{tab:cost_estimate} we show the runtime estimate which is computed in advance for each matrix, the total number of evaluations which is determined by this cost estimate, and the total runtime cost of these evaluations.}

\begin{table}
\revision{
    \centering
    \begin{tabular}{|c|c|c|c|}
    \hline
        Name & Single runtime estimate & Total evaluations & Total runtime\\ \hline
        add20 & 0.0342 & 2976 & 142.75 \\ \hline
        cavity05 & 0.1739 & 2976 & 367.53 \\ \hline
        cavity10 & 1.1450 & 528 & 390.35 \\ \hline
        cdde1 & 0.0095 & 2976 & 28.423 \\ \hline
        circuit\_2 & 0.1538 & 2976 & 521.12 \\ \hline
        fpga\_trans\_01 & 0.0418 & 2976 & 139.38 \\ \hline
        memplus & 1.7834 & 528 & 1195.1 \\ \hline
        orsirr\_1 & 0.0509 & 2976 & 181.50 \\ \hline
        orsreg\_1 & 0.0125 & 2976 & 62.667 \\ \hline
        pde2961 & 0.0202 & 2976 & 82.261 \\ \hline
        raefsky1 & 0.4391 & 2400 & 840.95 \\ \hline
        raefsky2 & 0.5579 & 2400 & 1302.7 \\ \hline
        rdb2048 & 0.0133 & 2976 & 48.420 \\ \hline
        sherman1 & 0.0334 & 2976 & 111.62 \\ \hline
        sherman4 & 0.0104 & 2976 & 31.089 \\ \hline
        sherman5 & 2.3297 & 528 & 1210.2 \\ \hline
        steam2 & 0.0682 & 2976 & 16.571 \\ \hline
        wang2 & 0.0772 & 2976 & 232.26 \\ \hline
        watt\_1 & 0.0502 & 2976 & 164.56 \\ \hline
        young3c & 0.0780 & 2976 & 260.78 \\ \hline \hline
        COMBINED & - & - & 7330.3 \\ \hline
    \end{tabular}
    \caption{\revision{The data presented is for the optimization for the training matrices using $m_{init} = 10$, summing total evaluations and runtime over all six steps of the three optimization cycles (cf. Figure~\ref{fig:optimization_sequence}). The single runtime estimate is computed in advance and based on it the structure of the quadtree is chosen which determines the total number of evaluations.\\
    Unlike all previous evaluations, the total runtime is the actual runtime of all completed PD-GMRES iterations and not the averaged heuristic which we use otherwise. This means that the total runtime is the actual time that the complete optimization procedure requires.}}
    \label{tab:cost_estimate}}
\end{table}

\end{document}